\documentclass[11pt,a4paper,fleqn]{article}
\usepackage{a4wide,amsfonts,amsmath,latexsym,amssymb,euscript,graphicx,units,mathrsfs,
hyperref}

\usepackage{graphicx}
\usepackage{color}
\usepackage{amssymb}
\usepackage{amssymb}
\usepackage[T1]{fontenc}
\usepackage{latexsym}
\usepackage{xypic}
\usepackage{eufrak}
\usepackage{euscript}
\usepackage{amsfonts,amsmath}
\usepackage{verbatim}
\usepackage{fancyhdr}
\usepackage[english]{babel}
\usepackage{mathrsfs}
\usepackage{units}

\newtheorem{prop}{Proposition}[section]
\newtheorem{cor}[prop]{Corollary}
\newtheorem{lemma}[prop]{Lemma}
\newtheorem{rem}[prop]{Remark}

\newtheorem{thm}[prop]{Theorem}
\newtheorem{defi}[prop]{Definition}
\newtheorem{example}[prop]{Example}

\renewcommand{\geq}{\geqslant}
\def\leq{\leqslant}

\newcommand{\R}{\mathbb{R}}

\def\HH{\EuFrak H}

\def\e{\varepsilon}

\def\1{{\mathbf{1}}}

\def\sk{{\mathbb{D}}}

\def\1{{\mathbf{1}}}
\def\0.5{{\frac{1}{2}}}

\newenvironment{proof}[1]{\begin{trivlist}\item {\it
\bf Proof.}\quad} {\qed\end{trivlist}}

\newcommand{\qed}{\nopagebreak\hspace*{\fill}
{\vrule width6pt height6ptdepth0pt}\par}

\begin{document}

\title{\textbf{ 
Entropy and the fourth moment phenomenon
}}
\author{Ivan Nourdin, Giovanni Peccati and Yvik Swan }
\maketitle
\abstract{We develop a new method for bounding the relative entropy of
  a  random vector in terms of its Stein
  factors.  Our approach is based on a novel representation
  for the score function of smoothly perturbed random variables, as
  well as on the de Bruijn's formula of information theory.  When applied to
  sequences of functionals of a general Gaussian field, our results
  can be combined with the Carbery-Wright inequality in order to
  yield multidimensional entropic rates of convergence that coincide,
  up to a logarithmic factor,
   with those achievable in smooth distances
  (such as the 1-Wasserstein distance). In particular, our findings
  settle the open problem of proving a quantitative version of the
  multidimensional {\it fourth moment theorem} for random vectors
  having chaotic components, with explicit rates of convergence in
  total variation that are independent of the order of the associated
  Wiener chaoses. The results proved in the present paper are outside
  the scope of other existing techniques, such as for instance the
  multidimensional Stein's method for normal approximations.

\


}

\tableofcontents
\section{Introduction}
\label{sec:introduction}

\subsection{Overview and motivation}

The aim of this paper is to develop a new method for controlling the
relative entropy of a general random vector with values in $\R^d$, and
then to apply this technique to settle a number of open questions
concerning central limit theorems (CLTs) on a Gaussian space. Our
approach is based on a fine analysis of the 
(multidimensional)
 {\it de
  Bruijn's formula}, which provides a neat representation
of the derivative of the relative entropy (along the
Ornstein-Uhlenbeck semigroup) in terms of the Fisher information of
some perturbed random vector -- see e.g. \cite{logsob_book,
  BakryEmery1985, BaBaNa03, Jo04, johnson2001entropy}. The main tool
developed in this paper (see Theorem~\ref{t:eulero} as well as relation
(\ref{key1})) is a new powerful representation of relative entropies
in terms of {\it Stein factors}. Roughly speaking, Stein factors
are random variables verifying a generalised integration by parts
formula (see (\ref{eq:26}) below): these objects naturally appear in
the context of the multidimensional {\it Stein's method} for normal
approximations (see e.g. \cite{ChGoSh11, NP11}), and implicitly play a
crucial role in many probabilistic limit theorems on Gaussian or other
spaces (see e.g. \cite[Chapter 6]{NP11}, as well as
\cite{nourdin2012absolute, NPRaop, NRaop}).

\medskip

The study of the classical CLT for sums of independent random elements
by entropic methods dates back to Linnik's seminal paper
\cite{LI59}. Among the many fundamental contributions to this line of
research, we cite \cite{artstein2004solution, ABBNptrf, BaBaNa03,
  BA86, MR2128239, BCG2012, Br82, carlen1991entropy, johnson2001entropy}
(see the monograph \cite{Jo04} for more details on the history of the
theory). All these influential works revolve around a deep
analysis of the effect of analytic convolution on the creation of
entropy: in this respect, a particularly powerful tool are the
`entropy jump inequalities' proved and exploited e.g. in
\cite{ABBNptrf, BaBaNa03, MR2128239, BaNg}. As discussed e.g. in
\cite{BaNg}, entropy jump inequalities are directly connected with
challenging open questions in convex geometry, like for instance the
{\it Hyperplane} and {\it KLS} conjectures. One of the common traits
of all the above references is that they develop tools to control the
Fisher information and  use the aforementioned de
Bruijn's formula to translate the bounds so obtained into bounds on
the relative entropy.

\medskip

One of the main motivations of the present paper is to initiate a
systematic information-theoretical analysis of a large class of CLTs
that has emerged in recent years in connection with different branches
of modern stochastic analysis. These limit theorems typically involve:
(a) an underlying infinite dimensional Gaussian field ${\bf G}$ (like for
instance a Wiener process), (b) a sequence of rescaled centered random
vectors $F_n = F_n({\bf G})$, $n\geq 1$, having the form of some highly
non-linear functional of the field ${\bf G}$. For example, each $F_n$ may be
defined as some collection of polynomial transformations of
${\bf G}$, possibly depending on a parameter that is integrated with respect
to a deterministic measure (but much more general forms are
possible). Objects of this type naturally appear e.g. in the
high-frequency analysis of random fields on homogeneous spaces
\cite{MP11}, fractional processes \cite{NouBookFBM, NuOrt}, Gaussian
polymers \cite{viensSPA}, or random matrices \cite{ChPtrf2009,
  NPalea}.

\medskip

In view of their intricate structure, it is in general not possible to
meaningfully represent the vectors $F_n$ in terms of some linear
transformation of independent (or weakly dependent) vectors, so that
the usual analytical techniques based on stochastic independence and
convolution (or mixing) cannot be applied. To overcome these
difficulties, a recently developed line of research (see \cite{NP11}
for an introduction) has revealed that, by using tools from
infinite-dimensional Gaussian analysis (e.g. the so-called {\it
  Malliavin calculus of variations} -- see \cite{Nu06}) and under some
regularity assumptions on $F_n$, one can control the distance between
the distribution of $F_n$ and that of some Gaussian target by means of
quantities that are no more complex than the fourth moment of
$F_n$. The regularity assumptions on $F_n$ are usually expressed in
terms of the projections of each $F_n$ on the eigenspaces of the
Ornstein-Uhlenbeck semigroup associated with ${\bf G}$ \cite{NP11, Nu06}.

\medskip

In this area, the most prominent contribution is arguably the
following one-dimensional inequality established by the first two
authors (see, e.g.,  \cite[Theorem 5.2.6]{NP11}): let $f$ be the density of a random variable $F$, assume
that $\int_\R x^2 f(x)dx=1$, and that $F$ belongs to the $q$th
eigenspace of the Ornstein-Uhlenbeck semigroup of 
${\bf G}$ 
(customarily
called the $q$th {\it Wiener chaos} of ${\bf G}$), then
\begin{equation}\label{e:nunugio}
\frac 14 \int_{\R} \left| f(x)  - \phi_1(x) \right| dx \leq  \sqrt{\frac{1}{3} - \frac{1}{3q}} \times \sqrt{\int_\R x^4(f(x)-\phi_1(x))dx},
\end{equation}
where $\phi_1(x) = (2\pi)^{-1/2} \exp( - x^2/2)$ is the standard Gaussian density (one can prove that $\int_\R x^4(f(x)-\phi_1(x))dx>0$ for $f$ as above). Note that $ \int_\R x^4\phi_1(x)dx = 3$.
A standard use of hypercontractivity therefore allows one to deduce the so-called {\it fourth moment theorem} established in \cite{NunuGio}: {\em for a rescaled sequence $\{F_n\}$ of random variables living inside the $q$th Wiener chaos of ${\bf G}$, one has that $F_n$ converges in distribution to a standard Gaussian random variable if and only if the fourth moment of $F_n$ converges to $3$ (and in this case the convergence is in the sense of total variation)}. See also \cite{NuOrt}.

\medskip

The quantity $\sqrt{\int_\R x^4(f(x)-\phi_1(x))dx}$ appearing in (\ref{e:nunugio}) is often called the {\it kurtosis} of the density $f$: it provides a rough measure of the discrepancy between the `fatness' of the tails of $f$ and $\phi_1$. The systematic emergence of the normal distribution from the reduction of kurtosis, in such a general collection of probabilistic models, is a new phenomenon that we barely begin to understand.  A detailed discussion of these results can be found in \cite[Chapters 5 and 6]{NP11}. M. Ledoux \cite{LedouxAOP} has recently proved a striking extension of the fourth moment theorem to random variables living in the eigenspaces associated with a general Markov operator, whereas references \cite{DNN, KNPS} contain similar statements in the framework of free probability. See also \cite{np-jfa, npr-jfa}, respectively, for a connection with second order Poincar\'e inequalities, and for a general analytical characterisation of cumulants on the Wiener space.
\medskip

The estimate (\ref{e:nunugio}) is obtained by combining the Malliavin calculus of
variations with the Stein's method for normal approximations
\cite{ChGoSh11, NP11}. Stein's method can be roughly described as a
collection of analytical techniques, allowing one to measure the
distance between random elements by controlling the regularity of the
solutions to some specific ordinary (in dimension 1) or partial (in
higher dimensions) differential equations. The needed estimates are
often expressed in terms of the same Stein factors that lie at the
core of the present paper (see Section \ref{ss:key} for
definitions). It is important to notice that the strength of these
techniques significantly breaks down when dealing with normal
approximations in dimension {\it strictly} greater than $1$. 

\medskip

For instance, in view of the
structure of the associated PDEs, for the time being there is no way
to directly use Stein's method in order to deduce bounds in the
multidimensional total variation distance. (This fact is demonstrated e.g. in references
\cite{MR2453473, MR2573554}, where multidimensional bounds are obtained for distances involving smooth test functions, as well as in \cite{Goetze}, containing a quantitative version of the classical multidimensional CLT, with bounds on distances involving test functions that are indicators of convex sets: all these papers use some variations of the multidimensional Stein's method, but none of them achieves bounds in the total variation distance.)

\medskip

In contrast, the results of
this paper allow one to deduce a number of information-theoretical
generalisations of (\ref{e:nunugio}) that are valid in any
dimension. It is somehow remarkable that our techniques make a
pervasive use of Stein factors, without ever applying Stein's
method. As an illustration, we present here a multidimensional
entropic fourth moment bound that will be proved in full generality in
Section \ref{s:gauss}. For $d\geq 1$, we write $\phi_d({\bf
  x})=\phi_d(x_1,...,x_d)$ to indicate the Gaussian density
$(2\pi)^{-d/2} \exp( - (x_1^2+\cdots + x_d^2)/2)$,
$(x_1,...,x_d)\in\R^d $.  From now on, every random object is assumed to be defined on a common probability space $(\Omega,
\mathcal{F}, P)$, with $E$ denoting expectation with respect to
$P$.

\begin{thm}[Entropic fourth moment bound]\label{t:e4m} Let $F_n = (F_{1,n },
  ..., F_{d,n})$ be a sequence of $d$-dimensional random vectors such that: {\rm(i)} $F_{i,n}$ belongs to the $q_i$th Wiener chaos of ${\bf G}$, with $1\leq q_1\leq q_2\leq \cdots \leq q_d$; {\rm (ii)}
  each 
  $F_{i,n}$ 
  has variance $1$, {\rm (iii)} $E[F_{i,n}F_{j,n}] = 0$ for $i\neq j$, and {\rm (iv)} the law of $F_n$ admits a density $f_n$
  on $\R^d$. Write
  \[
  \Delta_n :=\int_{\R^d} \|{\bf x}\|^4 (f_n({\bf x})-\phi_d({\bf x}))d{\bf x},
  \]
  where $\|\cdot\| $ stands for the Euclidean norm, and assume that $\Delta_n\to 0$, as $n\to\infty$. Then,  \begin{equation}\label{e:mp}
\int_{\R^d} f_n({\bf x})\log\frac{f_n({\bf x})}{\phi_d({\bf x})} d{\bf x} = O(1)\,  \Delta_n |\log\Delta_n|,
\end{equation}
where $O(1)$ stands for a bounded numerical sequence, depending on $d, q_1,...,q_d$ and on the sequence $\{F_n\}$.
\end{thm}

As in the one-dimensional case, one has always that $\Delta_n>0$ for
$f_n$ as in the previous statement. The quantity of the left-hand-side of (\ref{e:mp}) equals of course
the {\it relative entropy} of $f_n$. In view of the
Csiszar-Kullback-Pinsker inequality  (see  \cite{Cs, KullIEEE, PinskerBook}), according to which
\begin{equation}\label{ummaumma}
   \int_{\R^d} f_n({\bf x})\log\frac{f_n({\bf x})}{\phi_d({\bf x})} d{\bf x}\geq \frac12 \left(\int_{\R^d} \Big| f_n({\bf x}) -  \phi_d({\bf x})\Big| d{\bf x}\right)^2,
\end{equation}
relation (\ref{e:mp}) then translates in a bound on the square of the total
variation distance between $f_n$ and $\phi_d$, where the dependence in
$\Delta_n$ hinges on the order of the chaoses only via a
multiplicative constant. This bound agrees up to a logarithmic
factor with the estimates in smoother distances established in
\cite{NRaop} (see also \cite{NPRihp}), 
where it is proved that there exists a constant $K_0 =
K_0(d,q_1,...,q_d)$ such that
\[
{\bf W}_1(f_n,\phi_d) \leq K_0\,  \Delta_n^{1/2},
\]
where ${\bf W}_1$ stands for the usual Wasserstein distance of order 1. Relation (\ref{e:mp}) also drastically improves the bounds that can be deduced from \cite{nourdin2012absolute}, yielding that, as $n\to \infty$,
\[
\int_{\R^d} \Big| f_n({\bf x}) -  \phi_d({\bf x})\Big| d{\bf x} = O(1)\,  \Delta_n^{\alpha_d},
\]
where $\alpha_d$ is any strictly positive number verifying ${\alpha_d
  < \frac{1}{1+ (d+1)(3+4d(q_d-1))}}$, and the symbol $O(1)$ stands again for some bounded numerical sequence. The estimate (\ref{e:mp}) seems
to be largely outside the scope of any other available technique. Our
results will also show that convergence in relative entropy is a
necessary and sufficient condition for CLTs involving random vectors
whose components live in a fixed Wiener chaos. As in
\cite{nourdin2012absolute,nourdin2012convergence}, 
an important tool for establishing our
main results is the Carbery-Wright inequality \cite{CW}, providing
estimates on the small ball probabilities associated with polynomial
transformations of Gaussian vectors. Observe also that, via the
Talagrand's transport inequality \cite{TalagrandGFA}, our bounds
trivially provide estimates on the 2-Wasserstein distance
${\bf W}_2(f_n,\phi_d)$ between $f_n$ and $\phi_d$, for every $d\geq 1$. Notice once again that, in Theorem \ref{t:e4m} and its generalisations, {\it no additional regularity} (a part from the fact of being elements of a fixed Wiener chaos) is required from the components of the vector $F_n$. One should contrast this situation with the recent work by Hu, Lu and Nualart \cite{HLN}, where the authors achieve fourth moment bounds on the supremum norm of the difference $f_n - \phi_d$, under very strong additional conditions expressed in terms of the finiteness of negative moments of Malliavin matrices.

\medskip

We stress that, although our principal motivation comes from asymptotic
problems on a Gaussian space, the methods developed in Section
\ref{sec:prel-conc-inform} are general. In fact, at the heart of the
present work lie the powerful equivalences
\eqref{eq:17}--~\eqref{eq:django} (which can be considered as a new
form of so-called \emph{Stein identities}) that are valid under very
weak assumptions on the target density; it is also  easy to uncover a
wide variety of extensions and generalizations so that we
expect that our tools can be adapted to deal with a much
wider class of multidimensional distributions.

\medskip

The connection between  Stein identities and information theory has
already been noted in the  literature (although only in
dimension 1). For instance, explicit applications are known in the
context of Poisson and compound Poisson approximations
\cite{BaJoKoMa10, Sa12}, and recently several promising identities
have been discovered for some discrete
\cite{LS11c,sason2012entropy} as well as continuous distributions
\cite{kumar2009stein,LS12a,park2012equivalence}. However, with the
exception of \cite{LS12a}, the existing literature seems to be silent
about any connection between entropic CLTs and  Stein's identities  for
normal approximations. To the best of our
knowledge, together with \cite{LS12a} (which however focusses on
bounds of a completely different nature) the present paper contains
the first relevant study of the relations between the two topics.

\medskip

\noindent{\bf Remark on notation.} Given random vectors $X,Y$ with values in $\R^d$ ($d\geq 1$) and
densities $f_X, f_Y$, respectively, we
shall denote by ${\bf TV}(f_X, f_Y)$ and ${\bf W}_1(f_X,  f_Y)$ the {\it total
  variation} and 1-{\it Wasserstein} distances between $f_X$ and $f_Y$
(and thus  between the laws of $X$ and $Y$). Recall
that  we have the representations
\begin{align}
{\bf TV}(f_X, f_Y) &= \sup_{ A\in \mathscr{B}(\R^d) } \Big|
P[X\in A] -
P[Y\in A]\Big |\notag \\
&= \frac12 \sup_{\|h\|_\infty \leq 1 } \Big|
E [h(X)] - E[h(Y)]\Big | \label{umma}\\
&  = \frac12 \int_{\R^d} \Big| f_X({\bf x}) -  f_Y({\bf x})\Big| d{\bf x} =: \frac12 \| f_X - f_Y\|_1,\notag
\end{align}
where (here and throughout the paper) $d{\bf x}$ is shorthand for the
Lebesgue measure on $\R^d$, as well as
\begin{align*}
  {\bf W}_1(f_X, f_Y) &=\sup_{h\in {\rm Lip}(1)}
\Big| E[h(X) ] - E[h(Y)]\Big |.
\end{align*}
In order to simplify the discussion, we shall sometimes use the
shorthand notation 
\begin{align*}
  {\bf TV}(X, Y) = {\bf TV}(f_X, f_Y) \mbox{  and  } {\bf W}_1(X, Y) =
{\bf W}_1(f_X, f_Y).
\end{align*}
It is a well-known fact the the topologies induced by ${\bf TV}$ and
${\bf W}_1$, over the class of probability measures on $\R^d$, are strictly
stronger than the topology of convergence in distribution  (see
e.g. \cite[Chapter 11]{dudley_book} or \cite[Appendix
C]{NP11}). Finally, we agree that every logarithm in
the paper has base $e$.

\medskip

To enhance the readability of the text, the next Subsection \ref{ss:intromethod} contains an intuitive description of our method in dimension one.

\subsection{Illustration of the method in dimension one}\label{ss:intromethod}

\medskip

Let $F$ be a random variable with density $f:\R\to [0,\infty)$,
and let $Z$ be a standard Gaussian random variable with density
$\phi_1$. We shall assume that $E[F]=0$ and $E[F^2]=1$, and that $Z$
and $F$ are stochastically independent. As anticipated, we are
interested in bounding the relative entropy of $F$ (with respect to
$Z$), which is given by the quantity
\[
D(F\| Z) =  \int_\R f(x)\log(f(x)/\phi_1(x))dx.
\]
Recall also that, in view of the Pinsker-Csiszar-Kullback inequality, one has that
\begin{equation}\label{e:pinsker}
2{\bf TV}(f,\phi_1) \leq \sqrt{2 D(F || Z)}.
\end{equation}

\medskip

Our aim is to deduce a bound on $D(F\| Z)$ that is expressed in terms
of the so-called {\it Stein factor} associated with $F$. Whenever it exists, such a
factor is a mapping $\tau_F: \R \to \R$ that is uniquely determined (up to negligible sets)
 by requiring that  $\tau_F(F)\in L^1$ and 
\[
E[Fg(F)] = E[\tau_F(F) g'(F)]
\]
 for every smooth test function
$g$. Specifying $g(x) = x$ implies, in particular, that $E [\tau_F(F)]
= E [F^2] = 1$. It is easily seen
that, under standard regularity assumptions, a version of $\tau_F$ is
given by $\tau_F(x) =(f(x))^{-1} \int_x^\infty zf(z)dz$, for $x$ in the
support of $f$ (in particular, the Stein factor of $Z$ is 1). The relevance of the factor $\tau_F$ in comparing $F$
with $Z$ is actually revealed by the following {\it Stein's bound}
\cite{ChGoSh11, NP11}, which is one of the staples of Stein's method:
\begin{equation}\label{stein-ori}
{\bf TV}(f,\phi_1)= \sup \big|E[g'(F)]-E[Fg(F)]\big|,
\end{equation}
where the supremum runs over all continuously differentiable functions $g:\R\to\R$ satisfying $\|g\|_\infty\leq \sqrt{2/\pi}$ and
$\|g'\|_\infty\leq 2$. In particular, from (\ref{stein-ori}) one recovers the bound in total variation
 \begin{equation}\label{stein2}
 {\bf TV}(f,\phi_1)\leq 2\,E[|1-\tau_F(F)|],
 \end{equation}
providing a formal meaning to the intuitive fact that the
distributions of $F$ and $Z$ are close whenever $\tau_F$ is close to
$\tau_{Z}$, that is, whenever $\tau_F$ is close to 1. To motivate the reader, we shall now present a simple illustration of how the estimate (\ref{stein2}) applies to the usual CLT.
\begin{example} \label{ex:steinfactsum}
Let $\{F_i : i\geq 1\}$ be a sequence of i.i.d. copies of $F$, set $S_n = n^{-1/2}\sum_{i=1}^nF_i$ and assume that $E[\tau_F(F)^2] <+\infty$ (a simple sufficient condition for this to hold is e.g. that $f$ has compact support, and $f$ is bounded from below inside its support).  Then, using e.g. \cite[Lemma 2]{St86},
  \begin{equation}
    \label{eq:2}
    \tau_{S_n}(S_n) = \frac{1}{n}E \left[
    \sum_{i=1}^n \tau_{F}(F_i) \, \Big| \, S_n \right].
  \end{equation}
Since (by definition)  $E \left[
    \tau_{S_n}(S_n)  \right] = E \left[ \tau_{F}(F_i) \right]= 1$ for all
  $i=1, \ldots, n$ we get
  \begin{align}
    E \left[ (1- \tau_{S_n}(S_n))^2 \right] & = E \left[ \left(
        \frac{1}{n}E \left[ \sum_{i=1}^n (1-\tau_{F}(F_i)) \, \Big| \, F \right]\right)^2\right] \nonumber\\
&   \le   \frac{1}{n^2}E \left[
       \left( \sum_{i=1}^n (1-\tau_{F}(F_i)) \right)^2\right] \nonumber\\
& = \frac{1}{n^2} \mbox{\rm Var}\left(\sum_{i=1}^n (1-\tau_{F}(F_i))\right)
= \frac{ E \left[ (1-\tau_{F}(F))^2 \right]}{n}. \label{eq:3}
  \end{align}
In particular, writing $f_n$ for the density of $S_n$, we deduce from \eqref{stein2} that

\begin{equation}
  \label{eq:1}
  {\bf TV}(f_n, \phi_1) \le 2\frac{{\rm Var}(\tau_{F}(F))^{1/2}}{\sqrt n}.
\end{equation}

\end{example}

\smallskip

We shall demonstrate in Section~\ref{s:pregauss} and Section \ref{s:gauss} that the quantity
$E[|1-\tau_F(F)|]$ (as well as its multidimensional generalisations)
can   be explicitly controlled whenever $F$ is a smooth functional
of a Gaussian field. In view of these observations, the following
question is therefore natural: can one bound $D(F\| Z)$ by an
expression analogous to the right-hand-side of (\ref{stein2})?

\medskip

Our strategy for connecting $\tau_F(F)$ and $D(F\| Z)$ is based on an
integral version of the classical {\it de Bruijn's formula} of
information theory. To introduce this result, for $t\in [0,1]$ denote
by $f_t$ the density of $F_t = \sqrt{t} F + \sqrt{1-t} Z$, in such a
way that $f_1 = f$ and $f_0 = \phi_1$. Of course $f_t(x) = E \left[
  \phi_1 \left( (x- \sqrt{t}F)/{\sqrt{1-t}} \right)\right]/\sqrt{1-t}$
has support $\R$ and is $C^{\infty}$ for all $t<1$.  We shall denote
by $\rho_{t} = (\log f_t)'$ the {\it score function} of $F_t$ (which
is, by virtue of the preceding remark, well defined at all $t<1$
irrespective of the properties of $F$). For every $t<1$, the mapping
$\rho_t$ is completely characterised by the fact that
\begin{equation}\label{score}
E[g'(F_t)]=-E[g(F_t)\rho_t(F_t)]
\end{equation}
for every smooth test function $g$.
We also write, for $t\in[0,1)$,
\[
J(F_t)=E[\rho_t(F_t)^2]=\int_\R \frac{f_t'(x)^2}{f_t(x)}dx
\]
for the {\it Fisher information} of $F_t$, and we observe that
\[
0\leq E[(F_t+\rho_t(F_t))^2]=J(F_t)-1=:J_{st}(F_t),
\]
where $J_{st}(F_t)$ is the so-called {\it standardised Fisher information} of $F_t$ (note that $J_{st}(F_0) = J_{st}(Z)=0$). With this notation in mind, de Bruijn's formula (in an integral and rescaled version due to Barron \cite{BA86}) reads
\begin{equation}\label{key0}
D(F\|Z)=\int_0^1 \frac{J(F_t)-1}{2t}dt = \int_0^1 \frac{J_{st}(F_t)}{2t}dt
\end{equation}
(see Lemma \ref{l:db} below for a multidimensional statement).

\begin{rem}{
Using the standard relation $J_{st}(F_t)\leq t J_{st}(F) + (1-t)J_{st}(Z) = t J_{st}(F)$ (see e.g. \cite[Lemma 1.21]{Jo04}), we deduce the upper bound
\begin{equation}\label{shimizu}
D(F\|Z) \leq \frac12 J_{st}(F),
\end{equation}
a result which is often proved by using entropy power inequalities
(see also Shimizu \cite{Sh75}). Formula \eqref{shimizu} is a
quantitative counterpart to the intuitive fact that the distributions
of $F$ and $Z$ are close, whenever $J_{st}(F)$ is close to zero. Using
\eqref{e:pinsker} we further deduce that
closeness between the Fisher informations of $F$ and $Z$
(i.e. $J_{st}(F) \approx 0$) or between the entropies of $F$ and $Z$
(i.e. $D(F || Z) \approx 0$) both imply closeness in terms of the
total variation distance, and hence in terms of many more probability
metrics. This observation lies at the heart of the approach from
\cite{BA86,Br82,MR2128239} where a fine analysis of the behavior of
$\rho_F(F)$ over convolutions (through projection inequalities in the
spirit of \eqref{eq:2}) is used to provide explicit bounds on
the Fisher information distance which in turn are transformed, by
means of de Bruijn's identity \eqref{key0}, into bounds on the
relative entropy.  We will see in Section \ref{s:gauss} that the bound
\eqref{shimizu} is too crude to be of use in the applications we are
interested in.  }
\end{rem}

\medskip

Our key result in dimension 1 is the following statement (see Theorem
\ref{t:eulero} for a general multidimensional version), providing a
new representation of relative entropy in terms of Stein factors. From now on, we denote by $C^1_c$ the class of all functions $g:\R\to\R$ that are continuously differentiable and with compact support.

 \begin{prop}
Let the previous notation prevail. We have
\begin{equation}\label{key1}
D(F\|Z)=\frac12\int_0^1 \frac{t}{1-t}\,E\big[
E[Z(1-\tau_F(F))|F_t]^2
\big]
dt.
\end{equation}
\end{prop}
\begin{proof}{} Using  $\rho_Z(Z) = -Z$ we see that, for any function $g \in C^1_c$,
one has
\begin{align*}
&  E[Z(1-\tau_F(F))g(\sqrt{t}F+\sqrt{1-t}Z)]\\
&=\sqrt{1-t}\,E[(1-\tau_F(F))g'(\sqrt{t}F+\sqrt{1-t}Z)]\\
&=\sqrt{1-t}\big\{E[g'(F_t)]-\frac{1}{\sqrt{t}}
E[Fg(F_t)]\big\}\\
&=\frac{\sqrt{1-t}}{t}\big\{E[g'(F_t)]-
E[F_tg(F_t)]\big\}
=-\frac{\sqrt{1-t}}{t}\,E[(\rho_{t}(F_t)+F_t)g(F_t)],
\end{align*}
yielding the representation
\begin{equation}
  \label{eq:unidimrep}
  \rho_{t}(F_t)+F_t = -\frac{t}{\sqrt{1-t}}E[Z(1-\tau_F(F))|F_t].
\end{equation}
This implies
\[
J(F_t)-1 = E[(\rho_{t}(F_t)+F_t)^2] =
\frac{t^2}{1-t}E\big[
E[Z(1-\tau_F(F))|F_t]^2
\big],
\]
and the desired conclusion follows from  de Bruijn's identity (\ref{key0}).
\end{proof}

\bigskip

To properly control the integral on the right-hand-side of (\ref{key1}), we need
to deal with the fact that the mapping $t\mapsto \frac{t}{1-t}$ is not
integrable in $t=1$, so that we cannot directly apply the estimate $E\big[
E[Z(1-\tau_F(F))|F_t]^2
\big]\leq {\rm Var}(\tau_F(F))$ to deduce the desired bound. Intuitively, one has to exploit the fact that the mapping $t \mapsto E[Z(1-\tau_F(F))|F_t]$  satisfies
$E[Z(1-\tau_F(F))|F_1]=0$, thus in principle compensating for the singularity at $t \approx 1$. 

\medskip

As we will
see below, one can make this heuristic precise provided there
exist three constants $c,\delta,\eta>0$ such
that
\begin{equation}\label{guil}
E[|\tau_F(F)|^{2+\eta}]<\infty \mbox{   and   } E\big[
|E[Z(1-\tau_F(F))|F_t]|
\big]\leq c\,t^{-1}(1-t)^\delta,\quad 0<t\leq 1.
\end{equation}
Under the assumptions appearing in condition \eqref{guil},  the
following strategy can indeed be implemented in order to deduce a satisfactory bound.  First split the
integral in two parts: for every $0<\e\leq 1$,
\begin{align}
& 2\,D(F\|Z)\nonumber \\
& \leq
E[(1-\tau_F(F))^2]\,\int_0^{1-\e} \frac{t\,dt}{1-t}
+
\int_{1-\e}^1 \frac{t}{1-t}\,E\big[
E[Z(1-\tau_F(F))|F_t]^2
\big]
dt \nonumber \\
&\leq
E[(1-\tau_F(F))^2]\,|\log\e|
+
\int_{1-\e}^1 \frac{t}{1-t}\,E\big[
E[Z(1-\tau_F(F))|F_t]^2
\big]
dt,\label{eq:6}
\end{align}
the last inequality being a consequence of $\int_0^{1-\e}
\frac{t\,dt}{1-t}=\int_\e^{1} \frac{(1-u)du}{u}\leq \int_\e^{1}
\frac{du}{u}=-\log\e$.  To deal with the second term in \eqref{eq:6}, let us observe that, by using in particular the H\"older inequality
and the convexity of the function $x\mapsto |x|^{\eta+2}$, one deduces
from \eqref{guil} that
\begin{align}
&E\big[
E[Z(1-\tau_F(F))|F_t]^2
\big]\notag\\
&=
E\bigg[
|E[Z(1-\tau_F(F))|F_t]|^{\frac{\eta}{\eta+1}}
|E[Z(1-\tau_F(F))|F_t]|^{\frac{\eta+2}{\eta+1}}
\bigg]\notag\\
&\leq
E\big[
|E[Z(1-\tau_F(F))|F_t]|\big]^{\frac{\eta}{\eta+1}}\times
E\big[
|E[Z(1-\tau_F(F))|F_t]|^{\eta+2}\big]^{\frac{1}{\eta+1}}\notag\\
&\leq
c^{\frac{\eta}{\eta+1}}\,t^{-\frac{\eta}{\eta+1}}
(1-t)^{\frac{\delta\eta}{\eta+1}}
\times
E\big[|Z|^{\eta+2}\big]^{\frac{1}{\eta+1}}\times
E\big[|1-\tau_F(F)|^{\eta+2}\big]^{\frac{1}{\eta+1}}\notag\\
&\leq
c^{\frac{\eta}{\eta+1}}t^{-1}
(1-t)^{\frac{\delta\eta}{\eta+1}}\times 2\,
E\big[|Z|^{\eta+2}\big]^{\frac{1}{\eta+1}}\,
\big(1+E\big[|\tau_F(F)|^{\eta+2}\big]\big)^{\frac{1}{\eta+1}}\notag\\
&=C_{\eta}\,t^{-1}
(1-t)^{\frac{\delta\eta}{\eta+1}},\label{key3}
\end{align}
with
\[
C_\eta:=2c^{\frac{\eta}{\eta+1}}
E\big[|Z|^{\eta+2}\big]^{\frac{1}{\eta+1}}\,
\big(1+E\big[|\tau_F(F)|^{\eta+2}\big]\big)^{\frac{1}{\eta+1}}.
\]
By virtue of  \eqref{eq:6} and (\ref{key3}), the term $D(F\|Z)$ is eventually
amenable to analysis, and one obtains:
\begin{align*}
2\,D(F\|Z)&\leq
E[(1-\tau_F(F))^2]\,|\log\e|
+
C_\eta\int_{1-\e}^1 (1-t)^{\frac{\delta\eta}{\eta+1}-1}
dt\\
&=
E[(1-\tau_F(F))^2]\,|\log\e|
+
\frac{C_\eta(\eta+1)}{\delta\eta}\,\e^{\frac{\delta\eta}{\eta+1}}.
\end{align*}
Assuming finally that
$E[(1-\tau_F(F))^2]\leq 1$ (recall that, in the applications we are interested in, such a quantity is meant to be close to 0)
we can  optimize over $\e$ and choose $\e=
E[(1-\tau_F(F))^2]^{\frac{\eta+1}{\delta\eta}}$, which leads to
\begin{align}\notag
D(F\|Z)&\leq
\frac{\eta+1}{2\delta\eta}\,E[(1-\tau_F(F))^2]\,|\log E[(1-\tau_F(F))^2]| \\
&\quad+
\frac{C_\eta(\eta+1)}{2\delta\eta}\,E[(1-\tau_F(F))^2].
\label{keykey}
\end{align}
Clearly, combining (\ref{keykey}) with (\ref{e:pinsker}), one also obtains an
estimate in total variation which agrees with (\ref{stein2}) up to the
square root of a logarithmic factor.

\medskip

The problem is now how  to identify sufficient conditions on the
law of $F$ for (\ref{guil}) to hold; we shall address this issue by means of two
auxiliary results. We start with  a useful technical lemma, that has
been suggested to us by Guillaume Poly.

\begin{lemma}\label{lem:cooldual}
  Let $X$ be an integrable random variable and let $Y$ be a $\R^d$-valued random vector having an absolutely continuous distribution. Then
  \begin{equation}
    \label{eq:27}
E \left|  E \left[   X \, | \, Y \right]  \right| =
    \sup E \left[ X g(Y) \right],
  \end{equation}
where the supremum is taken over all $g\in C^1_c$ such that  $\|g \|_\infty\le 1$.
\end{lemma}
\begin{proof}{}
Since $|\mbox{sign}(E[X|Y])|= 1$ we have, by using e.g. Lusin's Theorem,
\begin{equation*}
  E \left|  E \left[   X \, | \, Y \right]  \right| = E[ X
\mbox{sign}(E[X|Y])] \leq \sup E (X g(Y)).
\end{equation*}
To see the reversed inequality, observe that, for any $g$ bounded by 1,
\begin{equation*}
  |E (X g(Y))|= | E (E(X | Y) g(Y)) | \leq  E \left|  E \left[   X \, | \, Y \right]  \right|.
\end{equation*}
The lemma is proved.
\end{proof}
Our next statement  relates
\eqref{guil} to the problem of estimating the total variation distance
between $F$ and $\sqrt{t}F+\sqrt{1-t}x$ for any $x\in\R$ and $0<t\leq
1$.

\begin{lemma}\label{sec:lemgui}
Assume that, for some $\kappa,\alpha>0$,
\begin{equation}
{\bf TV}(\sqrt t F + \sqrt{1-t}\,x,F)\label{ineq-ivan}
\leq \kappa(1+|x|)t^{-1}(1-t)^{\alpha},\quad x\in\R,\,t\in(0,1].
\end{equation}
Then (\ref{guil}) holds, with $\delta=\frac12\wedge\alpha$
and $c=4(\kappa+1)$.
\end{lemma}
\begin{proof}{}
Take $g\in  C^1_c$ such that
$\|g\|_{\infty}\leq 1$. Then, by
independence of $Z$ and $F$,
\begin{align*}
   E \left[  Z (1-\tau_F(F)) g(F_t)  \right]
& = E \left[ g(F_t)Z \right]- E \left[ Zg(F_t) \tau_F(F) \right]
\nonumber \\
& = E \left[ g(F_t)Z \right]-  \sqrt{1-t} E \left[\tau_F(F)g'(F_t) \right]
\nonumber \\
& =  E \left[Z(g(F_t)-g(F))\right] - \sqrt{\frac{1-t}{t}} E
\left[ g(F_t) F \right]
\end{align*}
so that, since $\|g\|_{\infty}\leq 1$ and $E|F|\leq \sqrt{E[F^2]}=1$,
\begin{align*}
 \left|  E \left[  Z (1-\tau_F(F))g(F_t)  \right] \right|  & \leq  \left|  E
 \left[Z \left( g(F_t)-g(F) \right)\right]   \right|+  \sqrt{\frac{1-t}{t}}\\
 &\leq\left|  E
 \left[Z \left(  g(F_t)-g(F) \right)\right]   \right|+  t^{-1}\sqrt{1-t}.
\end{align*}
We have furthermore
\begin{align*}
\big|E \left[Z(g(F_t)-g(F))\right]\big|
&=\left|\int_\R xE[g(\sqrt t F + \sqrt{1-t}\,x)-g(F)]\phi_1(x)dx\right|\\
&\leq
2\int_\R |x|\,\, {\bf TV}(\sqrt t F + \sqrt{1-t}\,x,F)\phi_1(x)dx\\
&\leq
2\kappa t^{-1}(1-t)^{\alpha}\int_\R |x| (1+|x|)\phi_1(x)dx\\
&\leq
4\kappa\,t^{-1}\,(1-t)^{\alpha}.
\end{align*}
Inequality (\ref{guil}) now follows by applying Lemma \ref{lem:cooldual}.
\end{proof}

\bigskip

As anticipated, in Section \ref{s:gauss} 
(see Lemma \ref{technical-but-nevertheless-crucial-lemma} for a precise statement) 
we will describe a wide class of distributions satisfying
\eqref{ineq-ivan}. The previous discussion yields
finally the following statement,
answering the original question of providing a bound on $D(F\|Z)$ that
is comparable with the estimate (\ref{stein2}).
\begin{thm}\label{main1d}
Let $F$ be a random variable with density $f:\R\to [0,\infty)$,
satisfying  $E[F]=0$ and $E[F^2]=1$.
Let $Z\sim \mathscr{N}(0,1)$ be a standard Gaussian variable
(independent of $F$).
If, for some $\alpha,\kappa,\eta>0$, one has
\begin{equation}
E[|\tau_F(F)|^{2+\eta}]<\infty\label{needed}
\end{equation}
and
\begin{equation}
{\bf TV}(\sqrt t F + \sqrt{1-t}\,x,F)\label{ineq-ivan2}
\leq \kappa(1+|x|)t^{-1}(1-t)^{\alpha},\quad x\in\R,\,t\in(0,1],
\end{equation}
then, provided $\Delta:=E[(1-\tau_F(F))^2]\leq 1$,
\begin{align}
&D(F\|Z)\leq
\frac{\eta+1}{(1\wedge 2\alpha)\eta}\,\Delta\,|\log \Delta|
+
\frac{C_\eta(\eta+1)}{(1\wedge 2\alpha)\eta}\,\Delta,
\label{yes}
\end{align}
where
\[
C_\eta =
2(4\kappa+4)^{\frac{\eta}{\eta+1}}
E\big[|Z|^{\eta+2}\big]^{\frac{1}{\eta+1}}\,
\big(1+E\big[|\tau_F(F)|^{\eta+2}\big]\big)^{\frac{1}{\eta+1}}.
\]
\end{thm}

\subsection{Plan}

The rest of the paper is organised as follows. In Section
\ref{sec:prel-conc-inform} we will prove
that Theorem \ref{main1d} can be generalised to a fully
multidimensional setting. Section \ref{s:pregauss} contains some general results related to (infinite-dimensional) Gaussian stochastic analysis. Finally, in Section~\ref{s:gauss} we shall apply our
estimates in order to deduce general bounds of the type appearing in
Theorem \ref{t:e4m}.

\section{Entropy bounds via de Bruijn's identity and Stein matrices}
\label{sec:prel-conc-inform}
In Section \ref{ss:ent} and Section \ref{ss:db} we discuss some
preliminary notions related to the theory of information (definitions,
notations and main properties). Section \ref{ss:key} contains the
proof of a new integral formula, allowing one to represent the
relative entropy of a given random vector in terms of a Stein
matrix. The reader is referred to the monograph \cite{Jo04}, as well
as to \cite[Chapter 10]{logsob_book}, for any unexplained definition
and result concerning information theory.

\subsection{Entropy}\label{ss:ent}

Fix an integer $d\geq 1$. Throughout this section, we consider a
$d$-dimensional square-integrable and centered random vector $F =
(F_1,...,F_d)$ with covariance matrix $B>0$. We shall assume that the
law of $F$ admits a density $f = f_F$ (with respect to the Lebesgue
measure) with support $S\subseteq \R^d$. 
No other assumptions on the
distribution of $F$ will be needed. We recall that the {\it
  differential entropy} (or, simply, the {\it entropy}) of $F$ is
given by the quantity $\mbox{Ent}(F) := -E[\log f(F)] = - \int_{\R^d} f({\bf
  x})\log f({\bf x}) d{\bf x} = - \int_{S} f({\bf x})\log f({\bf x})
d{\bf x}$, where we have adopted (here and for the rest of the paper) the standard convention $0\log0 := 0.$ Note that $\mbox{Ent}(F) = \mbox{Ent}(F+c)$ for all $c \in
\R^d$, i.e. entropy is location invariant.

As discussed above, we are interested in estimating the distance
between the law of $F$ and the law of a $d$-dimensional centered
Gaussian vector $Z = (Z_1,...,Z_d) \sim \mathscr{N}_d(0,C)$, where
$C>0$ is the associated covariance matrix. Our measure of the
discrepancy between the distributions of $F$ and $Z$ is the {\it
  relative entropy} (often called {\it Kullback-Leibler divergence} or
{\it information entropy})
 \begin{equation}
D(F || Z) := E \left[\log (f(F)/\phi(F)) \right] = \int_{\R^d} f({\bf x}) \log\left(\frac{f({\bf x})}{\phi({\bf x})}\right) d{\bf x},
\end{equation}
where $\phi = \phi_d(\cdot\,  ; C)$ is the density of
$Z$. It is easy to compute the Gaussian entropy $\mbox{Ent}(Z) = {1}/{2} \log
\left( (2\pi e)^d |C|\right)$ (where $|C|$ is the determinant of $C$),
from which we deduce the following alternative expression for the relative entropy
\begin{equation}
  \label{eq:14}
  0\leq   D(F || Z) = \mbox{Ent}(Z) - \mbox{Ent}(F) + \frac{ \mbox{tr}(C^{-1}B)-d}{2},
\end{equation}
where `$\mbox{tr}$' stands for the usual trace operator. If $Z$ and $F$ have the
same covariance matrix then the relative entropy is simply the entropy
gap between $F$ and $Z$ so that, in particular, one infers from \eqref{eq:14} that  $Z$ has maximal entropy among all absolutely
continuous  random vectors with covariance matrix $C$.

We
stress that the relative entropy $D$ does not define a \emph{bona
  fide} probability distance (for absence of a triangle inequality, as
well as for lack of symmetry): however, one can easily translate
estimates on the relative entropy in terms of the total variation
distance, using the already recalled Pinsker-Csiszar-Kullback inequality 
(\ref{ummaumma}).
In the next subsection,
we show how one can represent the quantity $D(F || Z)$ as the integral
of the standardized Fisher information of some adequate interpolation
between $F$ and $Z$.

\subsection{Fisher information and de Bruijn's identity}\label{ss:db}

Without loss of generality, we may assume for the rest of the paper
that the vectors $F$ and $Z$ (as defined in the previous Section
\ref{ss:ent}) are stochastically independent. For every $t\in [0,1]$,
we define the centered random vector $F_t := \sqrt{t} F +
\sqrt{1-t}Z$, in such a way that $F_0 = Z$ and $F_1 = F$. It is clear
that $F_t$ is centered and has covariance $\Gamma_t = t B +(1-t) C>0$;
moreover, whenever $t\in [0,1)$, $F_t$ has a strictly positive and
infinitely differentiable density, that we shall denote by $f_t$ (see
e.g. \cite[Lemma 3.1]{johnson2001entropy} for more details). For every
$t\in [0,1)$, we define the {\it score} of $F_t$ as the $\R^d$-valued
function given by
\begin{equation}\label{e:scoret1}
\rho_{t} : \R^d \to \R^d : {\bf x} \mapsto \rho_{t} ({\bf x}) =
(\rho_{t,1}({\bf x}),...,\rho_{t,d}({\bf x}))^T := \nabla\log f_t({\bf
  x}),
\end{equation}
with $\nabla$ the usual gradient in $\R^d$ (note that we will
systematically regard the elements of $\R^d$ as column vectors). The
quantity $\rho_{t} ({\bf x})$ is of course well-defined for every
${\bf x}\in \R^d$ and every $t\in[0,1)$; moreover, it is easily seen
that the random vector $\rho_t(F_t)$ is completely characterized (up
to sets of $P$-measure zero) by the relation
\begin{equation}\label{e:scoret2}
E[\rho_t(F_t) g(F_t)] = -E[\nabla g(F_t)],
\end{equation}
holding for every smooth function $g : \R^d \to \R$. Selecting $g=1$
in (\ref{e:scoret2}), one sees that $\rho_t(F_t)$ is a
centered random vector. The covariance matrix of $\rho_t(F_t)$ is
denoted by
\begin{equation}
\label{eq:7}
J(F_t) :=  E[\rho_t(F_t)\rho_t(F_t)^T]
\end{equation}
(with  components $J(F_t)_{ij}=E \left[ \rho_{t,i}(F_t)\rho_{t,j}(F_t) \right]$ for $1 \le
i, j \le d$), and is customarily called the {\it Fisher information
  matrix} of $F_t$. Focussing on the case $t=0$, one sees immediately
that the Gaussian vector $F_0 = Z \sim \mathscr{N}_d(0, C)$
has linear score function $\rho_0({\bf x}) = \rho_{Z}({\bf x}) = -C^{-1}{\bf x}$ and Fisher
information  $J(F_0) = J(Z) = C^{-1}$.

\begin{rem}{ Fix $t\in [0,1) $. Using formula (\ref{e:scoret2}) one
    deduces that a version of $\rho_t(F_t)$ is given by the
    conditional expectation $-(1-t)^{-1/2}E[C^{-1}Z | F_t]$, from
    which we infer that the matrix $J(F_t)$ is well-defined and its
    entries are all finite.}
\end{rem}

\smallskip

 For $t\in [0,1)$, we define the
{\it standardized Fisher information matrix} of $F_t$ as
\begin{equation}
  \label{eq:8}
  J_{st} (F_t) := \Gamma_t E \left[ \left( \rho_t(F_t) +
      \Gamma_t^{-1}F_t \right)\left( \rho_{t}(F_t) + \Gamma_t^{-1}F_t
    \right)^T 
\right] =  \Gamma_tJ(F_t) - I_d,
\end{equation}
where $I_d$ is the $d\times d$ identity matrix, and the last equality
holds because $E \left[ \rho_t(F_t)F_t \right] = -I_d$. Note that the 
positive semidefinite matrix $\Gamma_t^{-1} J_{st} (F_t) = J(F_t) -
\Gamma_t^{-1}$ is   the difference between the Fisher
information matrix of $F_t$ and that of a Gaussian vector having
distribution $\mathscr{N}_d(0,\Gamma_t)$. Observe that
\begin{equation}\label{e:equiv}
J_{st} (F_t) := E \left[ \left( \rho^{\star}_t(F_t) + F_t
  \right)\left( \rho^{\star}_{t}(F_t) + F_t
  \right)^T\right]\Gamma_t^{-1}, 
\end{equation}
where the vector
\begin{equation}\label{e:rhostar}
\rho^{\star}_t(F_t) = (\rho^{\star}_{t,1}(F_t),...,\rho^{\star}_{t,d}(F_t))^T := \Gamma_t \rho_t(F_t)
\end{equation}
 is completely characterized (up to sets of $P$-measure 0) by the equation
\begin{equation}\label{e:scoret3}
E[\rho^{\star}_t(F_t) g(F_t)] = -\Gamma_tE[\nabla g(F_t)],
\end{equation}
holding for every smooth test function $g$.

\begin{rem} \label{rem:fshfih} Of course the above information theoretic quantities are
  not defined only  for Gaussian
  mixtures of the form $F_t$ but more generally for any random vector
  satisfying the relevant assumptions (which are
  necessarily verified by the $F_t$). In particular, if $F$ has 
  covariance matrix $B$ and differentiable density $f$ then,  letting
  $\rho_F(\mathbf{x}) := \nabla \log f ({\mathbf x})$  be the score function for
  $F$,  the standardized Fisher information of $F$ is 
  \begin{equation}
    \label{eq:4}
    J_{st}(F) = B E \left[ \rho_F(F) \rho_F(F)^T \right].
  \end{equation}
In the event that the above be well-defined then  it is also 
scale invariant in the sense that $J_{st}(\alpha F) = J_{st}(F)$ for
all $\alpha \in \R$. 
\end{rem}

The following fundamental result is known as the (multidimensional)
{\it de Bruijn's identity}: it shows that the relative entropy
$D(F||Z)$ can be represented in terms of the integral of the mapping
$t \mapsto {\rm tr}(C\Gamma_t^{-1}J_{st}(F_t))$ with respect to the
measure $dt/2t$ on 
$(0,1]$. 
It is one of the staples of the entire
paper. We refer the reader e.g. to \cite{logsob_book, BA86} for proofs
in the case $d=1$. Our multidimensional statement is a rescaling of
\cite[Theorem 2.3]{johnson2001entropy} (some more details are given in
the proof). See also \cite{carlen1991entropy}.

\begin{lemma}[Multivariate de Bruijn's identity]\label{l:db} Let  the above notation and assumptions prevail. Then,
  \begin{align}
    \label{eq:debruijnD}
    D(F||Z) &= \int_0^1 \frac{1}{2t} {\rm tr} \left( C \Gamma_t^{-1}
 J_{st}(F_{t}) \right) dt  \\
 &\quad + \frac{1}{2} \left(  {\rm tr} \left( C^{-1}B \right)-d \right)
+ \int_0^1 \frac{1}{2t} {\rm tr} \left( C \Gamma_t^{-1}-I_d
\right) dt. \notag
\end{align}
 \end{lemma}
\begin{proof} {}In \cite[Theorem
2.3]{johnson2001entropy} it is proved that
\begin{align*}
  D(F||Z) &= \int_0^\infty \frac{1}{2} {\rm tr} \left( C (B+\tau C)^{-1} J_{st}(F+\sqrt{\tau}Z)) \right) d\tau
  \\
 &\quad + \frac{1}{2} \left(  {\rm tr} \left( C^{-1}B \right)-d \right)
+ \frac{1}{2}\int_0^\infty  {\rm tr} \left( C\left( (B+\tau C)^{-1} -\frac{C^{-1}}{1+\tau}
\right) \right) d\tau
\end{align*}
(note that the definition of standardized Fisher information used in
\cite{johnson2001entropy} is different from ours). The conclusion is
obtained by using the change of variables $t = (1+\tau)^{-1}$, as well
as the fact that 
\begin{align*}
  J_{st}\left(F+\sqrt{\frac{1-t}{t}} Z \right) = J_{st}\left( \sqrt{t} F+\sqrt{1-t} Z \right),
\end{align*}
which follows from the scale-invariance of standardized Fisher
information mentionned in Remark \ref{rem:fshfih}.
\end{proof}

\begin{rem}{ Assume that $C_n$, $n\geq 1$, is a sequence of $d\times d$ nonsingular covariance matrices such that $C_{n;i,j} \to B_{i,j} $ for every $i,j = 1,...,d$, as $n\to \infty$. Then, the second and third summands of
\eqref{eq:debruijnD} (with $C_n$ replacing $C$) converge to 0 as $n\to\infty$.

}
\end{rem}

For future reference, we will now rewrite formula (\ref{eq:debruijnD}) for some specific choices of $d$, $F$, $B$ and $C$.

\begin{example}{
\begin{itemize}
\item[(i)] Assume $F\sim \mathscr{N}_d(0,B)$. Then, $J_{st}(F_t) = 0$ (null matrix) for every $t\in [0,1)$, and formula (\ref{eq:debruijnD}) becomes
 \begin{equation}
    \label{eq:debruijnG}
    D(F||Z) = \frac{1}{2} \left(  {\rm tr} \left( C^{-1}B \right)-d \right)
+ \int_0^1 \frac{1}{2t} {\rm tr} \left( C \Gamma_t^{-1}-I_d
\right) dt.
\end{equation}
\item[(ii)]
Assume that $d=1$ and that $F$ and $Z$ have variances $b,c>0$, respectively. Defining $\gamma_t = tb + (1-t)c$, relation (\ref{eq:debruijnD}) becomes
\begin{align}
\label{eq:debruijn}
  D(F || Z) &= \int_0^1 \frac{c}{2 t\gamma_t} J_{st}(F_t)dt +\frac12 \left(\frac{b}{c} -1\right)+ \int_0^1 \frac{1}{2t} \left( \frac{c}{\gamma_t}-1 \right)dt\\
  & = \int_0^1 \frac{c}{2 t }E[(\rho_t(F_t) +\gamma_t^{-1}F_t)^2] dt  +\frac12 \left(\frac{b}{c} -1\right)+ \frac{\log c-\log b}{2}.\notag
\end{align}
Relation (\ref{eq:debruijn}) in the case $b=c$ $(= \gamma_t)$
corresponds to the integral formula \eqref{key0} proved by Barron in \cite[Lemma 1]{BA86}.
\item[(iii)] If $B=C$, then (\ref{eq:debruijnD}) takes the form
\begin{align}
    \label{eq:debruijnE}
D(F||Z) = \int_0^1 \frac{1}{2t} {\rm tr} \left(J_{st}(F_{t}) \right) dt.
\end{align}
In the special case where $B=C=I_d$, one has that
\begin{equation}
  \label{eq:30}
D(F || Z)  = \frac{1}{2} \sum_{j=1}^d \int_0^1 \frac{1}{t} E \left[ \left(
      \rho_{t,j}(F_t)+F_{t, j} \right)^2 \right]dt,
\end{equation}
of which \eqref{key0} is a particular case ($d=1$).
\end{itemize}

}
\end{example}

In the univariate setting, the general variance case ($E \left[ F^2
\right] = \sigma^2$) follows
trivially from the standardized one ($E \left[ F^2
\right] = 1$) through scaling; the same cannot be said in the
multivariate setting since the appealing form \eqref{eq:30} 
cannot be directly achieved for $d \ge 2$ when the
covariance matrices are not the identity because here  the dependence
structure of $F$ needs to be taken into account. In Lemma
\ref{lem:gencov} we provide an estimate
allowing one to deal with this difficulty in the case $B = C$, for
every $d$. The proof is based on the following elementary fact: if
$A,B$ are two $d\times d$ symmetric matrices, and if $A$ is
semi-positive definite, then
\begin{equation}\label{e:trick}
\lambda_{\min}(B) \times  {\rm tr}(A)\leq {\rm tr}(AB) \leq  \lambda_{\max}(B) \times  {\rm tr}(A),
\end{equation}
where $\lambda_{\min}(B)$ and $\lambda_{\max}(B)$ stand, respectively,
for the maximum and minimum eigenvalue of $B$. Observe that
$\lambda_{\max}(B) = \| B \|_{op}$, the operator norm of $B$.

\begin{lemma} \label{lem:gencov}Fix $d\geq 1$, and assume that $B=C$. Then, $C \Gamma_t^{-1} = I_d$, and one has the following estimates
\begin{align}\label{e:hi}
\lambda_{\min}(C)\times \sum_{j=1}^d E[ (\rho_{t,j}(F_t)
+(C^{-1}F_t)_j)^2  ] &\leq   {\rm tr} \left(J_{st}(F_{t}) \right)\\
\notag & \leq \lambda_{\max}(C)\times \sum_{j=1}^d E[ (\rho_{t,j}(F_t)
+(C^{-1}F_t)_j)^2  ], \\
\label{e:ho}  \lambda_{\min}(C^{-1}) \times \sum_{j=1}^d E[
(\rho^{\star}_{t,j}(F_t) + F_{t,j})^2  ] &\leq {\rm tr}  \left(
  J_{st}(F_{t}) \right) \\ \notag &\leq  \lambda_{\max}(C^{-1}) \times
\sum_{j=1}^d E[ (\rho^{\star}_{t,j}(F_t) + F_{t,j})^2  ].
\end{align}
\end{lemma}
\begin{proof} {}Write ${\rm tr} \left(J_{st}(F_{t}) \right) = {\rm tr}
  \left(C^{-1} J_{st}(F_{t}) C\right)$ and apply (\ref{e:trick}) first
  to $A =C^{-1} J_{st}(F_{t})$ and $B=C$, and then to $A =
  J_{st}(F_{t})C$ and $B=C^{-1}$.
\end{proof}

In the next section, we prove a new representation of the quantity $\rho_t(F_t) + C^{-1}F_t$ in terms of Stein matrices: this connection will provide the ideal framework in order to deal with the normal approximation of general random vectors.

\subsection{Stein matrices and a key lemma}\label{ss:key}

The centered $d$-dimensional vectors $F,Z$ are defined as in the
previous section (in particular, they are stochastically
independent). 

\begin{defi}[Stein matrices]\label{d:sm} {\rm Let $M(d,\R)$ denote the
    space of $d\times d$ real matrices. We say that the matrix-valued
    mapping 
$$
\tau_F : \R^d \to M(d, \R) : {\bf x} \mapsto \tau_F({\bf x}) = \{ \tau^{i,j}_F({\bf x}) : i,j=1,...,d \}
$$
is a {\it Stein matrix} for $F$  if $\tau_F^{i,j} (F)\in L^1$ for
every $i,j$ and   the following equality is verified for every
differentiable function $g : \R^d \to \R$ such that both sides are
well-defined: 
\begin{equation}
\label{eq:26}
E \left[ F g(F)  \right] = E \left[\tau_F(F)  \nabla g(F)  \right],
\end{equation}
or, equivalently,
\begin{equation}
\label{eq:26bis}
E \left[ F_i g(F) \right] = \sum_{j=1}^d E \left[ \tau^{i,j}_F(F) \partial_jg(F) \right], \quad i=1,...,d.
\end{equation}
The entries of the random matrix  $\tau_F(F)$ are called the  {\it Stein
  factors} of $F$.}
\end{defi}

\begin{rem}{

\begin{itemize}
\item[(i)] Selecting $g(F) = F_j$ in (\ref{eq:26}), one deduces that, if $\tau_F$ is a Stein matrix for $F$, then
  $E[\tau_F(F)] = B$. More to this point, if $F\sim \mathscr{N}_d(0, B)$, then the covariance matrix
  $B$ is itself a Stein matrix for $F$. This last relation is known as the {\it Stein's
    identity} for the multivariate Gaussian distribution.

\item[(ii)] Assume that $d=1$ and that $F$ has density $f$ and
  variance $b>0$. Then, under some standard regularity assumptions, it
  is easy to see
  that $\tau_F(x) =  b \int_x^{\infty}y f(y) dy/f(x)$ is a Stein
  factor for $F$.
\end{itemize}
}
\end{rem}

\begin{lemma}[Key Lemma] \label{l:key} Let the above notation and
  framework prevail, and assume that $\tau_F$ is a Stein matrix for
  $F$ such that $\tau_F^{i,j}(F)\in L^1(\Omega)$ for every
  $i,j=1,...,d$. Then, for every $t\in [0 ,1)$, the mapping
  \begin{equation}
    \label{eq:17}
    {\bf x}\mapsto   -\frac{t}{\sqrt{1-t}}E\Big[  \big(I_d- C^{-1}\tau_F(F) \big) C^{-1}Z
\,\, \Big| \,\, F_t = {\bf x}  \Big]- C^{-1}{\bf x}
 \end{equation}
is a version of the score $\rho_{t}$ of $F_t$. Also, the mapping
\begin{equation}
    \label{eq:django}
    {\bf x}\mapsto   -\frac{t}{\sqrt{1-t}}E\Big[  \big(\Gamma_t- \Gamma_t C^{-1}\tau_F(F) \big) C^{-1}Z
\,\, \Big| \,\, F_t = {\bf x}  \Big]- \Gamma_t C^{-1}{\bf x}
 \end{equation}
 is a version of the function $\rho^{\star}_t $ defined in formula (\ref{e:rhostar}).
\end{lemma}
\begin{proof}{} Remember that $-C^{-1} Z$ is the score of $Z$, and denote by ${\bf x}\mapsto A_t({\bf x})$ the mapping defined in (\ref{eq:17}). Removing the conditional expectation and exploiting the independence of $F$ and $Z$, we infer that, for every smooth test function $g$,
\begin{align*}
  E \left[ A_t(F_t) g(F_t) \right]  & =  \frac{t}{\sqrt{1-t}}
  E\Big[  \big(I_d- C^{-1}\tau_F(F) \big)\left( -C^{-1}Z \right)
  g(F_t) \Big] \\
&  \quad\quad-     \sqrt t C^{-1} E \left[ F g(F_t) \right] - \sqrt{1-t} E
\left[ C^{-1}Z g(F_t) \right]\\
 & = - t  E \left[ \big(I_d- C^{-1} \tau_F(F) \big)  \nabla g(F_t)
\right] \\
&\quad\quad  -  t  C^{-1}  E \left[ \tau_F(F) \nabla g(F_t) \right] -
(1-t) E \left[ \nabla g(F_t) \right]\\
& = - E \left[ \nabla g(F_t) \right],
\end{align*}
thus yielding the desired conclusion.
\end{proof}

To simplify the forthcoming discussion, we shall use the shorthand
notation: $\widetilde{Z} =(\widetilde{Z}_1,...,\widetilde{Z}_d) :=
C^{-1}Z \sim \mathscr{N}_d(0, C^{-1})$, $\widetilde{F}
=(\widetilde{F}_1,...,\widetilde{F}_d) := C^{-1}F$, and
$\widetilde{\tau}_F =\{\widetilde{\tau}^{i,j}_F : i,j=1,...,d\} :=
C^{-1}\tau_F$. The following statement is the main achievement of the
section, and is obtained by combining Lemma \ref{l:key} with formulae
(\ref{eq:debruijnE}) and (\ref{e:hi})--(\ref{e:ho}), in the case where
$C=B$.

\begin{thm}\label{t:eulero} Let the above notation and assumptions
  prevail, assume that $B=C$, and introduce the notation
\begin{align}
A_1(F;Z) &:= \frac12 \int_{0}^1 \frac{t}{1-t} E\left[\sum_{j=1}^d
  E\left[\sum_{k=1}^d ({\bf 1}_{j=k} - \widetilde{\tau}_F^{j,k}(F))
    \widetilde{Z}_k \,\, \Big|  \,\, F_t
  \right]^2\right]dt,  \label{e:sa}\\
A_2(F;Z) &:= \frac12 \int_{0}^1 \frac{t}{1-t} E\left[ \sum_{j=1}^d
  E\left[\sum_{k=1}^d (C(j,k) - {\tau}_F^{j,k}(F)) \widetilde{Z}_k
    \,\, \Big|  \,\, F_t\right]^2\right]dt. \label{e:vall}
\end{align}
Then one has the inequalities
\begin{align}
&\lambda_{\min}(C) \times A_1(F;Z) \leq D(F\| Z) \leq \lambda_{\max}(C)
\times A_1(F;Z), \label{e:euler1}\\
&\lambda_{\min}(C^{-1}) \times A_2(F;Z) \leq D(F\| Z) \leq
\lambda_{\max}(C^{-1}) \times A_2(F;Z)\label{euler2}.
\end{align}
In particular, when $C=B=I_d$,
\begin{equation}\label{e:euler3}
  D(F\| Z) =  \frac12 \int_{0}^1 \frac{t}{1-t} E\left[ \sum_{j=1}^d
    E\left[\sum_{k=1}^d ({\bf 1}_{j=k} -{\tau}_F^{j,k}(F)) {Z}_k \,\,
      \Big|  \,\, F_t  \right]^2\right]dt.
\end{equation}
\end{thm}

The next subsection focusses on general bounds based on the estimates \eqref{e:sa}--\eqref{e:euler3}.

\subsection{A general bound}

The following statement provides the announced multidimensional generalisation of Theorem \ref{main1d}. In particular, the main estimate (\ref{wished}) provides an explicit quantitative counterpart to the heuristic fact that, if there exists a Stein matrix $\tau_F$ such that $\|\tau_F-C\|_{H.S.}$ is small (with $\|\cdot\|_{H.S.}$ denoting the usual Hilbert-Schmidt norm), then the distribution of $F$ and $Z\sim \mathscr{N}_d(0,C)$ must be close. By virtue of Theorem \ref{t:eulero}, the proximity of the two distributions is expressed in terms of the relative entropy $D(F\| Z)$.

\begin{thm}\label{main-multi}
Let $F$ be a centered and square integrable random vector
with density $f:\R^d\to [0,\infty)$, let $C>0$ be its covariance matrix, and assume that $\tau_F$ is a Stein matrix for $F$.
Let $Z\sim \mathscr{N}_d(0,C)$ be a Gaussian random vector
independent of $F$.
If, for some $\kappa,\eta>0$ and $\alpha\in(0,\frac12\big]$, one has
\begin{equation}
E\big[|\tau_F^{j,k}(F)|^{\eta+2}\big]<\infty,\quad j,k=1,\ldots,d,\label{needed-multi}
\end{equation}
as well as
\begin{equation}
{\bf TV}\Big(\sqrt t F + \sqrt{1-t}\,{\bf x},F\Big)\label{ineq-ivan2-multi}
\leq \kappa(1+\|{\bf x}\|_1)(1-t)^{\alpha},\quad {\bf x}\in\R^d,\,t\in\left[1/2,1\right],
\end{equation}
then, provided
\begin{equation}\label{lessthanone}
\Delta:= E[\| C- \tau_F\|_{H.S.}^2] = \sum_{j,k=1}^d E\left[\left(C(j,k)-\tau_F^{j,k}(F)\right)^2\right]\leq 2^{-\frac{\eta+1}{\alpha\eta}},
\end{equation}
one has
\begin{align}
D(F\|Z)&\leq
\frac{d(\eta+1)\lambda_{\max}(C^{-1})}{2\alpha\eta}\max_{1\leq l\leq
  d}E[(\widetilde{Z}_l)^2]\times \Delta\,|\log \Delta| \nonumber \\
& \quad +
\frac{C_{d,\eta,\tau}(\eta+1)\lambda_{\max}(C^{-1})}{2\alpha\eta}\Delta,
\label{wished}
\end{align}
where
\begin{align}\notag
C_{d,\eta,\tau}&:= 2d^2\left(2\kappa\,E[\|Z\|_1(1+\|Z\|_1)]+ \sqrt{\max_j C(j,j)}\right)\max_{1\leq l\leq d}E[|\widetilde{Z}_l]^{\eta+2}]^{\frac{1}{\eta+1}}\\
&\hskip2cm\times
\sum_{j,k=1}^d
\left(|C(j,k)|^{\eta+2}+E\big[|\tau_F^{j,k}(F)|^{\eta+2}\big]\right)^{\frac{1}{\eta+1}},
\label{cdeta}
\end{align}
and (as above) $\widetilde{Z} = C^{-1} Z$.
\end{thm}
\begin{proof}

Take $g\in  C^1_c$ such that
$\|g\|_{\infty}\leq 1$. Then, by
independence of $\widetilde{Z}$ and $F$
and using (\ref{eq:26bis}), one has, for any $j=1,\ldots,d$,
\begin{align*}
 &  E \left[  \sum_{k=1}^d (C(j,k)-\tau_F^{j,k}(F))\widetilde{Z}_k\,g(F_t)  \right] \nonumber\\
& = \sum_{k,l=1}^d C(j,k)C^{-1}(k,l)\,E \left[Z_l\, g(F_t) \right]- \sum_{k,l=1}^d  C^{-1}(k,l) E \left[  \tau_F^{j,k}(F) Z_l \,g(F_t)\right]
\nonumber \\
& = E \left[Z_j\, g(F_t) \right]-
\sqrt{1-t} \sum_{k,l,m=1}^d   C^{-1}(k,l)C(l,m)  E \left[  \tau_F^{j,k}(F) \partial_m g(F_t)\right]
\nonumber \\
& = E \left[Z_j\, (g(F_t)-g(F)) \right]-
\sqrt{1-t} \sum_{k=1}^d    E \left[  \tau_F^{j,k}(F) \partial_k g(F_t)\right]
\nonumber \\
& = E \left[Z_j\, (g(F_t)-g(F)) \right]-
\sqrt{\frac{1-t}{t}} E \left[  F_j g(F_t)\right].
\nonumber
\end{align*}
Using (\ref{ineq-ivan2-multi}), we have
\begin{align*}
& \left|E \left[Z_j(g(F_t)-g(F)) \right]\right| \\
&=\left|\int_{\R^d} x_j E[g(\sqrt t F + \sqrt{1-t}\,{\bf x})-g(F)]\,\frac{e^{-\frac12\langle C^{-1}{\bf x},{\bf x}\rangle}}{(2\pi)^{\frac{d}{2}}\sqrt{\det C}}\,d{\bf x}\right|\\
&\leq
2\int_{\R^d} |x_j| {\bf TV}(\sqrt t F + \sqrt{1-t}\,{\bf x},F)\,\frac{e^{-\frac12\langle C^{-1}{\bf x},{\bf x}\rangle}}{(2\pi)^{\frac{d}{2}}\sqrt{\det C}}\,d{\bf x}\\
&\leq
2\kappa (1-t)^{\alpha}\int_{\R^d} \|{\bf x}\|_1 (1+\|{\bf x}\|_1)\,\frac{e^{-\frac12\langle C^{-1}{\bf x},{\bf x}\rangle}}{(2\pi)^{\frac{d}{2}}\sqrt{\det C}}\,d{\bf x}\\
&=
2\kappa\,E[\|Z\|_1(1+\|Z\|_1)]\,(1-t)^{\alpha}.
\end{align*}
As a result, due to Lemma \ref{lem:cooldual} and since $E|F_j|\leq \sqrt{E[F_j^2]}\leq \sqrt{\max_j C(j,j)}$,
one obtains
\begin{align*}
&\max_{1\leq j\leq d}E\left[ \left|   E \left[ \sum_{k=1}^d (C(j,k)-\tau_F^{j,k}(F))\widetilde{Z}_k\bigg| F_t  \right] \right|\right]\\
&\leq \left(2\kappa\,E[\|Z\|_1(1+\|Z\|_1)]+ \sqrt{\max_j C(j,j)}\right)(1-t)^{\alpha}.
\end{align*}
Now, using among others the H\"older inequality
and the convexity of the function $x\mapsto |x|^{\eta+2}$, we have that
\begin{align}
&\sum_{j=1}^dE\left[
 E \left[ \sum_{k=1}^d (C(j,k)-\tau_F^{j,k}(F))\widetilde{Z}_k\bigg| F_t  \right]^2
\right]\nonumber\\
&=\sum_{j=1}^d
E\left[
\left| E \left[ \sum_{k=1}^d (C(j,k)-\tau_F^{j,k}(F))\widetilde{Z}_k\bigg| F_t  \right]\right|^{\frac{\eta}{\eta+1}}
 \right.\times\nonumber \\
&\quad\quad\quad\quad\times \left.\left| E \left[ \sum_{k=1}^d (C(j,k)-\tau_F^{j,k}(F))\widetilde{Z}_k\bigg| F_t  \right]\right|^{\frac{\eta+2}{\eta+1}}
\right]\nonumber\\
&\leq
\sum_{j=1}^d
E\left[
\left| E \left[ \sum_{k=1}^d (C(j,k)-\tau_F^{j,k}(F))\widetilde{Z}_k\bigg| F_t  \right]\right|\right]^{\frac{\eta}{\eta+1}}\times\nonumber \\
&\quad\quad\quad\quad\times E\left[\left| E \left[ \sum_{k=1}^d (C(j,k)-\tau_F^{j,k}(F))\widetilde{Z}_k\bigg| F_t  \right]\right|^{\eta+2}
\right]^{\frac{1}{\eta+1}}\notag\\
&\leq
C_{d,\eta,\tau} \,\,
(1-t)^{\frac{\alpha\eta}{\eta+1}},\label{key3multi}
\end{align}
with $C_{d,\eta,\tau}$ given by (\ref{cdeta}).
At this stage, we shall use the key identity (\ref{euler2}).
To properly control the right-hand-side of (\ref{e:vall}), we split the
integral in two parts: for every $0<\e\leq \frac12$,
\begin{align*}
&\frac{2}{\lambda_{max}(C^{-1})}\,D(F\|Z)\nonumber\\
& \leq
\sum_{j=1}^d
 E \left[ \left(\sum_{k=1}^d (C(j,k)-\tau_F^{j,k}(F))\widetilde{Z}_k\right)^2
\right]\,\int_0^{1-\e} \frac{t\,dt}{1-t}\nonumber \\
&\quad +
\int_{1-\e}^1 \frac{t}{1-t}\,\sum_{j=1}^dE\left[
 E \left[ \sum_{k=1}^d (C(j,k)-\tau_F^{j,k}(F))\widetilde{Z}_k\bigg| F_t  \right]^2
\right]
dt \nonumber \\
&\leq
d\max_{1\leq l\leq d}E[(\widetilde{Z}_l)^2]\,\sum_{j,k=1}^d
E\left[\left(C(j,k)-\tau_F^{j,k}(F)\right)^2\right]\,|\log\e|\nonumber
\\ 
& \quad+C_{d,\eta,\tau}
\int_{1-\e}^1
(1-t)^{\frac{\alpha\eta}{\eta+1}-1}
dt\nonumber\\ 
&\leq
d\max_{1\leq l\leq d}E[(\widetilde{Z}_l)^2]\,\sum_{j,k=1}^d E\left[\left(C(j,k)-\tau_F^{j,k}(F)\right)^2\right]\,|\log\e|
+C_{d,\eta,\tau}
\frac{\eta+1}{\alpha\eta}\,\e^{\frac{\alpha\eta}{\eta+1}},
\end{align*}
the second inequality being a consequence of (\ref{key3multi}) as well as $\int_0^{1-\e}
\frac{t\,dt}{1-t}=\int_\e^{1} \frac{(1-u)du}{u}\leq \int_\e^{1}
\frac{du}{u}=-\log\e$.
Since  (\ref{lessthanone}) holds true,
one can  optimize over $\e$ and choose
$\e=
\Delta^{\frac{\eta+1}{\alpha\eta}}$,
which leads to the desired estimate (\ref{wished}).
\end{proof}

\subsection{A general roadmap for proving quantitative entropic CLTs}

In Section \ref{s:gauss}, we will apply the content of Theorem \ref{main-multi} to deduce entropic fourth moment theorems on a Gaussian space. As demonstrated in the discussion to follow, this achievement is just one possible application of a general strategy, leading to effective entropic estimates by means of Theorem \ref{main-multi}. 

\medskip

\noindent{\bf Description of the strategy}. Fix $d\geq 1$, and consider a sequence $F_n = (F_{1,n},...,F_{d,n})$, $n\geq 1$, of centered random vectors with covariance $C_n>0$ and such that $F_n$ has a density for every $n$. Let $Z_n \sim \mathscr{N}_d(0, C_n)$, $n\geq 1$, be a sequence of Gaussian vectors, and assume that $C_n \to C>0$, as $n\to \infty$. Then, in order to show that $D(F_n \| Z_n) \to 0$ one has to accomplish the following steps:

\begin{itemize}

\item[(a)] {\it Write explicitly the Stein matrix $\tau_{F_n}$ for
    every $n$}. This task can be realised either by exploiting the
  explicit form of the density of $F_n$, or by applying integration by
  parts formulae, whenever they are available. This latter case often
  leads to a representation of the Stein matrix by means of a
  conditional expectation. As an explicit example and as shown in Section
  \ref{ss:mat} below, Stein matrices on a Gaussian space can be easily
  expressed in terms of Malliavin operators.

\item[(b)] {\it Check that, for some $\eta>0$, the quantity
    $E\big[|\tau_{F_n}^{j,k}(F_n)|^{\eta+2}\big]$ is bounded by a
    finite constant independent of $n$}. This step typically requires
  one to show that  the sequence $\{ \tau_{F_n}^{j,k}(F_n) : n\geq
  1\}$ is bounded in $L^2(\Omega)$, and moreover that the elements of
  this sequence verify some adequate hypercontractivity property. We
  will see in Section \ref{s:gauss} that, when considering random
  variables living inside a finite sum of Wiener chaoses, these
  properties are implied by the explicit representation of Stein
  matrices in terms of Malliavin operators, as well as by the
  fundamental inequality stated in Proposition~\ref{P : Hyper}.

\item[(c)] {\it Prove a total variation bound such as
    \eqref{ineq-ivan2-multi} for every $F_n$, with a constant $\kappa$
    independent of $n$}. This is arguably the most delicate step, as
  it consists in an explicit estimate of the total variation distance
  between the law of $F_n$ and that of $\sqrt{t}F_n + \sqrt{1-t} \,
  {\bf x}$, when $t$ is close to one. It is an interesting remark
  that, for every $i=1,...,d$ (as $t\to 1$), $E[(F_{i,n} -
  (\sqrt{t}F_{i,n} + \sqrt{1-t} \, x_i ))^2] = O(1-t)$, so that an
  estimate such as \eqref{ineq-ivan2-multi} basically requires one to
  explicitly relate the total variation and mean-square distances
  between $F_n$ and $\sqrt{t}F_n + \sqrt{1-t} \, {\bf x}$. We will see
  in Section \ref{s:gauss} that, for random variables living inside a
  fixed sum of Wiener chaoses, adequate bounds of this type can be
  deduced by applying the Carbery-Wright inequalities
  \cite{CW}, that we shall combine with the approach developed by
  Nourdin, Nualart and Poly in \cite{nourdin2012absolute}.

\item[(d)] {\it Show that $ E\| \tau_{F_n} - C_n\|^2 \to 0$}. This is realised by using the explicit representation of the Stein matrices $\tau_{F_n}$.

\end{itemize} 

Once steps (a)--(d) are accomplished, an upper bound on the speed of convergence to zero of the sequence $D(F_n \| Z_n)$ can be deduced from \eqref{wished}, whereas the total variation distance between the laws of $F_n$ and $Z$ follows from the inequality
\[
{\bf TV} (F_n,Z) \leq {\bf TV} (F_n,Z_n) + {\bf TV} (Z ,Z_n) \leq   \frac{1}{\sqrt{2}}\left[ \sqrt{D(F_n\| Z_n)} + \sqrt{D(Z\|Z_n)}\right].
\]
Note that the quantity $D(Z\|Z_n)$ can be assessed by using \eqref{eq:debruijnG}.

\medskip

{\color{red}Before proceeding to the application of the above roadmap
  on Gaussian space we now first provide, in  the forthcoming Section
  \ref{s:pregauss}, the necessary
preliminary results about Gaussian stochastic analysis.}

\section{Gaussian spaces and variational calculus}\label{s:pregauss}

As announced, we shall now focus on random variables that can be written as functionals of a countable collection of independent and identically distributed Gaussian $\mathscr{N}(0,1)$ random variables, that we shall denote by
\begin{equation}\label{e:gaussf}
{\bf G}=\Big\{G_i : i\geq 1\Big\}.
\end{equation}
Note that our description of ${\bf G}$ is equivalent to saying that ${\bf G}$ is a Gaussian sequence such that $E[G_i] = 0$ for every $i$ and $E[G_iG_j] = {\bf 1}_{\{i=j\}}$. We will write $L^2(\sigma({\bf G})):= L^2(P,\sigma({\bf G}))$ to indicate the class of square-integrable (real-valued) random variables that are measurable with respect to the $\sigma$-field generated by ${\bf G}$.

\smallskip

The reader is referred e.g. to \cite{NP11, Nu06} for any unexplained definition or result appearing in the subsequent subsections.

\subsection{Wiener chaos}
We will now briefly introduce the notion of \textit{Wiener chaos}.

\begin{defi}[Hermite polynomials and Wiener chaos]\label{Def HER +
    Chaos} 
\

{\begin{itemize}
\item[1.] The sequence of {\it Hermite polynomials} $\{H_m : m\geq 0\}$ is defined as follows: $H_0 =1$, and, for $m\geq 1$,
\[
H_m (x) = (-1)^{m}e^{\frac{x^2}{2}}\frac{d^m}{dx^m} e^{-\frac{x^2}{2}}, \quad x\in \R.
\]
It is a standard result that the sequence $\{(m!)^{-1/2}H_m:m\geq 0\}$ is an orthonormal basis of $L^2(\R, \phi_1(x)dx)$.
\item[2.] A {\it multi-index} $\alpha = \{\alpha_i : i\geq 1 \}$ is a sequence of nonnegative integers such that $\alpha_i \neq 0$ only for a finite number of indices $i$. We use the symbol $\Lambda$ in order to indicate the collection of all multi-indices, and use the notation $|\alpha| = \sum_{i\geq 1} \alpha_i$, for every $\alpha\in\Lambda$.
\item[3.] For every integer $q\geq 0$, the $q$th {\it Wiener chaos} associated with ${\bf G}$ is defined as follows: $C_0 = \R$, and, for $q\geq 1$, $C_q$ is the $L^2(P)$-closed vector space generated by random variables of the type
    \begin{equation}\label{Eq : Q-chaos}
    \Phi(\alpha)=\prod^{\infty}_{i=1} H_{\alpha_i}(G_i), \quad \alpha\in\Lambda\,\,\, \text{and} \,\,\,|\alpha| = q.
    \end{equation}
\end{itemize}
}
\end{defi}

It is easily seen that two random variables belonging to Wiener chaoses of different orders are orthogonal in $L^2(\sigma({\bf G}))$. Moreover, since linear combinations of polynomials are dense in $L^2(\sigma({\bf G}))$, one has that $L^2(\sigma({\bf G}))=\bigoplus_{q\geq 0} C_q$, that is, any square-integrable functional of ${\bf G}$ can be written as an infinite sum, converging in $L^2$ and such that the $q$th summand is an element of $C_q$. This orthogonal decomposition of $L^2(\sigma({\bf G}))$ is customarily called the {\it Wiener-It\^{o} chaotic decomposition} of $L^2(\sigma({\bf G}))$. 

\smallskip

It is often convenient to encode random variables in the spaces $C_q$ by means of increasing
tensor powers of Hilbert spaces. To do this, introduce 
an 
(arbitrary) separable real Hilbert space $\mathfrak{H}$ having an orthonormal basis $\{e_i : i\geq 1\}$. For $q\geq 2$, denote
by $\mathfrak{H}^{\otimes q}$ (resp. $\mathfrak{H}^{\odot q}$)  the $q$th tensor power (resp. symmetric
tensor power) of $\mathfrak{H}$; write moreover $\mathfrak{H}^{\otimes 0} =\mathfrak{H}^{\odot 0}=
\R$ and $\mathfrak{H}^{\otimes 1} =\mathfrak{H}^{\odot 1}=
\mathfrak{H}$. With every multi-index
$\alpha\in \Lambda$, we associate the tensor $e(\alpha)\in\mathfrak{H}^{\otimes |\alpha|}$ given by
\[
e(\alpha) = e_{i_1}^{\otimes \alpha_{i_1}}\otimes \cdot\cdot\cdot\otimes e_{i_k}^{\otimes \alpha_{i_k}},
\]
where $\{\alpha_{i_1},...,\alpha_{i_k}\}$ are the non-zero elements of $\alpha$. We also denote by $\tilde{e}(\alpha)\in
\mathfrak{H}^{\odot |\alpha|}$ the canonical symmetrization of $e(\alpha)$. It is well-known that, for every $q\geq 2$, the set $\{\tilde{e}(\alpha):\alpha\in\Lambda, \, |\alpha|=q\}$ is a complete orthogonal system in $\mathfrak{H}^{\odot  q}$. For every $q\geq 1$ and every $h\in\mathfrak{H}^{\odot  q}$ of the form $h = \sum_{\alpha\in\Lambda, \, |\alpha|=q}c_{\alpha}\tilde{e}(\alpha),$ we define
\begin{equation}\label{Eq : ISO}
I_q(h) = \sum_{\alpha\in\Lambda, \, |\alpha|=q}c_{\alpha}\Phi(\alpha),
\end{equation}
where $\Phi(\alpha)$ is given in (\ref{Eq : Q-chaos}). Another classical result (see e.g. \cite{NP11, Nu06}) is that, for every $q\geq 1$, the mapping $I_q : \mathfrak{H}^{\odot  q}\rightarrow C_q $ (as defined in (\ref{Eq : ISO})) is onto, and defines an isomorphism between $C_q$ and the Hilbert space $\mathfrak{H}^{\odot  q}$, endowed with the modified norm $\sqrt{q!}\|\cdot\|_{\mathfrak{H}^{\otimes  q}}$. This means that, for every $h,h'\in \mathfrak{H}^{\odot  q}$, $E[I_q(h)I_q(h')] = q! \langle h , h' \rangle_{\mathfrak{H}^{\otimes  q}}.$ 

\smallskip

Finally, we observe that one can reexpress the Wiener-It\^o chaotic decomposition of $L^2(\sigma({\bf G}))$ as follows: every $F\in L^2(\sigma({\bf G}))$ admits a unique decomposition of the type 
\begin{equation}\label{ChaosExpansion}
F=\sum_{q=0}^\infty I_q(h_q),
\end{equation}
where the series converges in $L^2({\bf G})$,
the symmetric kernels $h_q\in\HH^{\odot q}$, $q\geq1$, are
uniquely determined by $F$, and $I_0(h_0) := E[F]$. This also implies that \[E[F^2] = E[F]^2+  \sum_{q=1}^\infty q!\|h_q\|^2_{\HH^{\otimes q}}.\]

\subsection{The language of Malliavin calculus: chaoses as eigenspaces}

We let the previous notation and assumptions prevail: in particular, we shall fix for the rest of the section a real separable Hilbert space $\HH$, and represent the elements of the $q$th Wiener chaos of ${\bf G}$ in the form (\ref{Eq : ISO}). In addition to the previously introduced notation, $L^2(\HH):=L^2(\sigma({\bf G});\HH) $ indicates the space of all $\HH$-valued random elements $u$,
that are measurable with respect to $\sigma({\bf G})$ and verify the relation $E\big[\|u\|_\HH ^2\big]<\infty$. Note that, as it is customary, $\HH$ is endowed with the Borel $\sigma$-field associated with the distance on $\HH$ given by $(h_1,h_2)\mapsto \|h_1-h_2\|_\HH$.

\smallskip

Let $\mathscr{S}$ be the set of all smooth cylindrical random
variables of the form 
\[
F = g\big(I_1(h_1), \ldots,
I_1(h_n)\big),
\] 
where $n\geq 1$, $g : \R^n \rightarrow \R$ is a
smooth function with compact support and $h_i\in\HH$. The
{\rm Malliavin derivative} of $F$ (with respect to ${\bf G}$) is the element of $L^2(\HH)$ defined as
\[
DF\; =\; \sum_{i =1}^n \partial_i g\big(I_1(h_1), \ldots, I_1(h_n)\big)
h_i.
\]
By
iteration, one can define the $m$th derivative $D^m F$ (which is
an element of ${L}^2 (\HH^{\otimes m})$) for every
$m\geq 2$. For $m\geq 1$, $\sk^{m,2}$ denotes the
closure of $\mathscr{S}$ with respect to the norm $\| \cdot
\|_{m,2}$, defined by the relation
\[
\| F\|_{m,2}^2  =  E[F^2] +\sum_{i=1}^m
E\big[ \| D^i F\|_{\HH^{\otimes i}}^2\big].
\]
It is a standard result that a random variable $F$ as in \eqref{ChaosExpansion} is in $ \sk^{m,2}$ if and only if $\sum_{q\geq 1} q^mq! \|f_q\|^2_{\HH^{\otimes q}}<\infty$, from which one deduces that $\bigoplus_{k=0}^q C_k\in \sk^{m,2}$ for every $q,m\geq 1$. Also, $DI_1(h) = h$ for every $h\in \HH$.
The Malliavin derivative $D$ satisfies the following \textit{chain
rule}: if $g:\R^n\rightarrow\R$ is continuously differentiable
and has bounded partial derivatives,
and if
$(F_1,...,F_n)$ is a vector of elements of $\sk^{1,2}$,
then $g(F_1,\ldots,F_n)\in\sk^{1,2}$ and
\begin{equation}\label{chainR}
Dg(F_1,\ldots,F_n)=\sum_{i=1}^n
\partial_i g(F_1,\ldots, F_n)DF_i.
\end{equation}

\smallskip

In what follows, we denote by $\delta$ the
adjoint of the operator $D$, also called the \textit{divergence
operator}. A random element $u \in L^{2}(\HH)$
belongs to the domain of $\delta$, written ${\rm Dom}\,\delta$, if
and only if it satisfies
\[
| E \left[ \langle D F,u\rangle_{\EuFrak H} \right]|\leq c_u\,\sqrt{E[F^2]}\quad\mbox{for any }F\in{\mathscr{S}},
\]
for some constant $c_u$ depending only on $u$. If $u \in
{\rm Dom}\,\delta $, then the random variable $ \delta(u)$ is
defined by the duality relationship (customarily called
``integration by parts formula''):
\begin{equation}\label{ipp}
E \left[ F \delta(u) \right]=  E \left[  \langle D F, u \rangle_{\HH} \right],
\end{equation}
which holds for every $F \in \sk^{1,2}$. A crucial object for our discussion is the Ornstein-Uhlenbeck semigroup associated with ${\bf G}$.

\begin{defi}[Ornstein-Uhlenbeck semigroup]{\rm Let ${\bf G}'$ be an independent copy of ${\bf G}$, and denote by $E'$ the mathematical expectation with respect to ${\bf G}'$. For every $t\geq 0$ the operator $P_t : L^2(\sigma({\bf G})) \to L^2(\sigma({\bf G}))$ is defined as follows: for every $F({\bf G}) \in L^2(\sigma({\bf G}))$,
\[
P_t F({\bf G}) = E'[F(e^{-t}{\bf G} + \sqrt{1-e^{-2t}}{\bf G}')],
\]
in such a way that $P_0F({\bf G}) = F({\bf G})$ and $P_\infty F({\bf G})= E[F({\bf G})]$. The collection $\{P_t : t\geq 0\}$ verifies the semigroup property $P_{t}P_s = P_{t+s}$ and is called the {\it Ornstein-Uhlenbeck semigroup} associated with ${\bf G}$.}
\end{defi}

The properties of the semigroup $\{P_t : t\geq 0\}$ that are relevant for our study are gathered together in the next statement.

\begin{prop} \begin{enumerate}

\item For every $t>0$, the eigenspaces of the operator $P_t$ coincide with the Wiener chaoses $C_q$, $q=0,1,...$, the eigenvalue of $C_q$ being given by the positive constant $e^{-qt}$.

\item The {\rm infinitesimal generator} of $\{P_t:t\geq 0\}$, denoted by $L$, acts on square-integrable random variables as follows: a random variable $F$ with the form (\ref{ChaosExpansion}) is in the domain of $L$, written ${\rm Dom}\, L$, if and only if $\sum_{q\geq 1}
qI_q(h_q)$ is convergent in $L^2(\sigma({\bf G}))$, and in this case
\[ 
LF = -\sum_{q\geq 1}qI_q(h_q).
\] 
In particular, each Wiener chaos $C_q$ is an eigenspace of $L$, with eigenvalue equal to $-q$.

\item The operator $L$ verifies the following
properties: {\rm (i)} ${\rm Dom}\, L=\sk^{2,2}$, and {\rm (ii)} a random variable $F$ is in ${\rm Dom} L$ if and only if
$F\in{\rm Dom}\,\delta D$ (i.e.
$F\in\sk^{1,2}$ and $DF\in{\rm Dom}\delta$), and in this case one has that
$\delta (DF)=-LF$.

\end{enumerate}

\end{prop}

In view of the previous statement, it is immediate to describe the
\textit{pseudo-inverse} of $L$, denoted by $L^{-1}$, as follows: for every mean zero random variable $F=\sum_{q\geq 1}I_q(h_q)$ of $L^2(\sigma({\bf G}))$,
one has that 
\begin{equation*}
  L^{-1}F =\sum_{q\geq 1}-\frac{1}{q} I_q(h_q).
\end{equation*}
It is clear that
$L^{-1}$ is an operator with values in $\sk^{2,2}$.

\smallskip

For future reference, we record the following estimate involving random variables living in a finite sum of Wiener chaoses: it is a direct consequence of the hypercontractivity of the Ornstein-Uhlenbeck semigroup -- see e.g. in \cite[Theorem 2.7.2 and Theorem 2.8.12]{NP11}.

\begin{prop}[Hypercontractivity]\label{P : Hyper} Let $q\geq 1$ and $1\leq  s < t <\infty$. Then, there exists a finite constant $c(s,t,q)<\infty$ such that, for every $F\in \bigoplus^q_{k=0} C_k$,
\begin{equation}\label{EQ : hyperC}
E[|F|^t]^{\frac 1t} \leq c(s,t,q)\,E[|F|^s]^{\frac 1s}.
\end{equation}
In particular, all $L^p$ norms, $p\geq 1$, are equivalent on a finite sum of Wiener chaoses.
\end{prop}

Since we will systematically work on a fixed sum of Wiener chaoses, we will not need to specify the explicit value of the constant $c(s,t,q)$. See again \cite{NP11}, and the references therein, for more details.

\begin{example} Let $W =\{W_t : t\geq 0\}$ be a standard Brownian motion, let $\{e_j : j\geq 1\}$ be an orthonormal basis of $L^2(\R_+, \mathscr{B}(\R_+), dt)=:L^2(\R_+)$, and define $G_j = \int_0^\infty e_j(t)dW_t$. Then, $\sigma(W) = \sigma({\bf G})$, where ${\bf G} = \{G_j : j\geq 1\}$. In this case, the natural choice of a Hilbert space is $\HH = L^2(\R_+)$ and one has the following explicit characterisation of the Wiener chaoses associated with $W$: for every $q\geq 1$, one has that $F\in C_q$ if and only if there exists a symmetric kernel $f\in L^2(\R_+^q)$ such that 
\[
F = q! \int_0^\infty \int_0^{t_1} \cdots \int_0^{t_{q-1}} f(t_1,...,t_q) dW_{t_q}\cdots dW_{t_1} :=q! J_q(f).
\]
The random variable $J_q(f)$ is called the {\rm multiple Wiener-It\^o integral} of order $q$, of $f$ with respect to $W$. It is a well-known fact that, if $F\in \sk^{1,2}$ admits the chaotic expansion $F = E[F] +\sum_{q\geq 1} J_q(f_q)$, then $DF$ equals the random function
\begin{align*}
t&\mapsto \sum_{q\geq 1} qJ_{q-1}(f_q(t,\cdot)) \\ 
& = \sum_{q\geq 1} q! \int_0^\infty \int_0^{t_1} \cdots \int_0^{t_{q-2}} f(t,t_1...,t_{q-1}) dW_{t_{q-1}}\cdots dW_{t_1}, \quad t\in \R_+,
\end{align*}
which is a well-defined element of $L^2(\HH)$.
\end{example}

\subsection{The role of Malliavin and Stein matrices}\label{ss:mat}

Given a vector $F = (F_1,...,F_d)$ whose elements are in $\mathbb{D}^{1,2}$, we define the {\it Malliavin matrix } $\Gamma(F) =\{\Gamma_{i,j}(F) : i,j=1,...,d\}$ as
\[
\Gamma_{i,j}(F) = \langle DF_i, DF_j\rangle_\HH.
\]
The following statement is taken from \cite{nourdin2012absolute}, and provides a simple necessary and sufficient condition for a random vector living in a finite sum of Wiener chaoses to have a density.

\begin{thm}\label{t:nnp} Fix $d,q\geq 1$ and let $F = (F_1,...,F_d)$ be such that $F_i\in \bigoplus_{i=0}^qC_i$, $i=1,...,d$. Then, the distribution of $F$ admits a density $f$ with respect to the Lebesgue measure on $\R^d$ if and only if $E[{\rm det}\, \Gamma(F)]>0$. Moreover, if this condition is verified one has necessarily that ${\rm Ent}(F)<\infty$. 
\end{thm}
\begin{proof}

The equivalence between the existence of a density and the fact that
$E[{\rm det}\, \Gamma(F)]>0$ is shown in \cite[Theorem 3.1]{nourdin2012absolute}.
Moreover, by \cite[Theorem 4.2]{nourdin2012absolute} we have
in this case that the density $f$ satisfies $f\in\bigcup_{p>1}L^p(\R^d)$.
Relying on the inequality
\[
\log u\leq n(u^{1/n}-1),\quad u>0,\quad n\in\mathbb{N}
\]
(which is a direct consequence of $\log u\leq u-1$ for any $u>0$),
one has that
\[
\int_{\R^d} f({\bf x})\log f({\bf x})d{\bf x}\leq n\left(\int_{\R^d}f({\bf x})^{1+\frac1n}d{\bf x} - 1\right).
\]
Hence, by choosing $n$ large enough so that $f\in L^{1+\frac{1}{n}}(\R^d)$,
one obtains that  ${\rm Ent}(F)<\infty$. 
\end{proof}

To conclude, we present a result providing an explicit representation for Stein matrices associated with random vectors in the domain of $D$.

\begin{prop}[Representation of Stein matrices]\label{p:smg} Fix  $d\geq 1$ and let $F = (F_1,...,F_d)$ be a centered random vector whose elements are in $\mathbb{D}^{1,2}$. Then, a Stein matrix for $F$ (see Definition \ref{d:sm}) is given by
\[
\tau_F^{i,j}({\bf x}) = E[\langle -DL^{-1}F_i, DF_j\rangle_\HH \big | F={\bf x}], \quad i,j=1,...,d.
\]
\end{prop}
\begin{proof}
{}Let $g:\R^d \to \R \in C^1_c$. For every $i=1,...,d$, one has that $F_i = -\delta DL^{-1} F_i$. As a consequence, using (in order) (\ref{ipp}) and (\ref{chainR}), one infers that
\[
E[F_i g(F)] = E[\langle -DL^{-1}F_i, Dg(F)\rangle_\HH] = \sum_{j=1}^d E[\langle -DL^{-1}F_i, DF_j\rangle_\HH\,\, \partial_j g(F)].
\]
Taking conditional expectations yields the desired conclusion.
\end{proof}

The next section contains the statements and proofs of our main bounds on a Gaussian space.

\section{Entropic fourth moment bounds on a Gaussian space}\label{s:gauss}

\subsection{Main results}

We let the framework of the previous section prevail: in particular,
${\bf G}$ is a sequence of i.i.d. $\mathscr{N}(0,1)$ random variables,
and the sequence of Wiener chaoses $\{C_q : q\geq 1\}$ associated with
${\bf G}$ is encoded by means of increasing tensor powers of a fixed
real separable Hilbert space $\HH$. 
We will
use the following notation: given a sequence of centered and
square-integrable $d$-dimensional random vectors $F_n =
(F_{1,n},...,F_{d,n})\in \sk^{1,2}$ with covariance matrix  $C_n$,
$n\geq 1$, we let
\begin{equation}\label{e:deltaF}
\Delta_n := E[\|F_n\|^4] - E[\|Z_n\|^4],
\end{equation}
where $\| \cdot \|$ is the Euclidian norm on $\R^d$ and $Z_n\sim \mathscr{N}_d(0, C_n)$.

\medskip

Our main result is the following entropic central limit theorem for sequences of chaotic random variables.

\begin{thm}[Entropic CLTs on Wiener chaos]\label{t:e4m-gene} Let $d\geq 1$ and $q_1,\ldots,q_d\geq 1$ be fixed integers.
Consider vectors
\[
F_n = (F_{1,n},\ldots, F_{d,n}) = (I_{q_1}(h_{1,n}),\ldots, I_{q_d}(h_{d,n})),\quad n\geq 1,
  \]
with $h_{i,n}\in\HH^{\odot q_i}$.
Let $C_n$ denote the covariance matrix of $F_n$ and let $Z_n\sim\mathscr{N}_d(0,C_n)$
be a centered Gaussian random vector in $\R^d$ with the same covariance matrix as $F_n$. Assume that $C_n\to C>0$ and $\Delta_n\to 0$, as $n\to\infty$. Then, the random vector $F_n$ admits a density for $n$ large enough, and
\begin{equation}\label{entropy}
D(F_n\|Z_n) = O(1)\,\Delta_n |\log\Delta_n|\quad\mbox{as $n\to\infty$},
\end{equation}
where $O(1)$ indicates a bounded numerical sequence depending on $d, q_1,...,q_d$, as well as on the sequence $\{F_n\}$.
\end{thm}

One immediate consequence of the previous statement is the following characterisation of entropic CLTs on a finite sum of Wiener chaoses.

\begin{cor}\label{c:co} Let the sequence $\{F_n\}$ be as in the statement of Theorem \ref{t:e4m-gene}, and assume that $C_n \to C>0$. Then, the following three assertions are equivalent, as $n\to \infty$:
\begin{itemize}
\item[\rm (i)] $\Delta_n \to 0$ ;

\item[\rm (ii)] $F_n$ converges in distribution to $Z\sim \mathscr{N}_d(0,C)$;

\item[\rm (iii)] $D(F_n\|Z_n)\to 0$.

\end{itemize}

\end{cor}

\medskip

The proofs of the previous results are based on two technical lemmas that are the object of the next section.

\subsection{Estimates based on the Carbery-Wright inequalities}

We start with a generalisation of the inequality proved by Carbery and Wright in \cite{CW}. Recall that, in a form that is adapted to our framework, the main finding of \cite{CW} reads as follows: there is a universal constant $c>0$ such that, for any polynomial $Q:\mathbb{R}^n \rightarrow \mathbb{R}$ of degree at most $d$ and  any $\alpha>0$ we have
\begin{equation}  \label{ff1}
E[Q(X_1, \dots, X_n)^2] ^{\frac 1{2d}} \,\,P(|Q(X_1, \dots, X_n)| \leq \alpha) \leq cd\alpha^{\frac 1d},
\end{equation}
where $X_1, \dots, X_n$ are independent random variables with common distribution $\mathscr{N}(0,1)$.

\begin{lemma}\label{carbery-wright}
Fix $d,q_1,\ldots,q_d\geq 1$, and let 
$F=(F_{1},\ldots,F_{d})$ be a random vector such that
$F_{i}=I_{q_i}(h_i)$ with $h_i\in \HH^{\odot q_i}$. Let $\Gamma =
\Gamma(F)$ denote the Malliavin matrix of $F$, and assume that $E[\det
\Gamma]>0$ (which is equivalent to assuming that $F$ has a density  
by Theorem \ref{t:nnp}).
Set $N=2d(q-1)$
with $q=\max_{1\leq i\leq d} q_i$.
Then, there exists a universal constant $c>0$ such that
\begin{equation}\label{carbery}
P(\det \Gamma \leq \lambda)\leq cN\lambda^{1/N}(E[\det\Gamma])^{-1/N}.
\end{equation}
\end{lemma}
{\it Proof}. Let $\{e_i: i\geq 1\} $ be an orthonormal basis of $\mathfrak{ H}$. Since $\det \Gamma$ is a polynomial of degree $d$ in the entries of $\Gamma$ and because each entry of $\Gamma$ belongs to
$\bigoplus_{k=0}^{2q-2}C_k$ by the product formula for multiple integrals (see, e.g., \cite[Chapter 2]{NP11}), we have, by iterating the product formula, that 
$\det\Gamma \in \bigoplus_{k=0}^NC_k$.
Thus, there exists a sequence $\{Q_n, n\geq 1\}$ of real-valued polynomials of degree at most $N$ such that the random variables
$Q_n(I_1(e_1),\ldots,I_1(e_n))$
converge in $L^2$ and almost surely to $\det \Gamma$ as $n$ tends to infinity (see \cite[proof of Theorem 3.1]{nourdin2012convergence} for an explicit construction).
Assume now that $E[\det \Gamma]>0$. Then, for $n$ sufficiently large,
$E[|Q_n(I_1(e_1),\ldots,I_1(e_n))|]>0$.
We deduce from the estimate (\ref{ff1})  the existence of a universal constant $c>0$ such that, for any $n\geq 1$,
\[
P(|Q_n(I_1(e_1),\ldots,I_1(e_n)| \leq \lambda)\leq cN\lambda^{1/N}(E[Q_n(I_1(e_1),\ldots,I_1(e_n)^2])^{-1/2N}.
\]
Using the property
\[
E[Q_n(I_1(e_1),\ldots,I_1(e_n) ^2]\geq (E[|Q_n(I_1(e_1),\ldots,I_1(e_n)| ])^2
\]
we obtain
\[
P(|Q_n(I_1(e_1),\ldots,I_1(e_n)| \leq \lambda)\leq cN\lambda^{1/N}(E[| Q_n(I_1(e_1),\ldots,I_1(e_n))|])^{-1/N},
\]
from which (\ref{carbery}) follows by letting $n$ tend to infinity.
\qed

\medskip

The next statement, whose proof follows the same lines than that of
\cite[Theorem 4.1]{nourdin2012absolute}, provides an upper bound on
the total variation distance between the distribution of $F$ and that of
$\sqrt{t}F+\sqrt{1-t}{\bf x}$, for every ${\bf x}\in\R^d$ and every $t\in[1/2;
1]$. Although Lemma~\ref{technical-but-nevertheless-crucial-lemma} is,
in principle, very close to \cite[Theorem 4.1]{nourdin2012absolute}, we
detail its proof because, here, we need to keep track of the way the
constants behave with respect to ${\bf x}$.

\begin{lemma}\label{technical-but-nevertheless-crucial-lemma}
Fix $d,q_1,\ldots,q_d\geq 1$, and let $F=(F_{1},\ldots,F_{d})$ be a random vector as in Lemma \ref{carbery-wright}.
Set $q=\max_{1\leq i\leq d} q_i$.
 Let $C$ be the covariance matrix of $F$,
let $\Gamma = \Gamma(F)$ denote the Malliavin matrix of $F$, and assume that $\beta:=E[\det\Gamma]>0$.
Then, there exists a constant $c_{q,d,\|C\|_{H.S.}}>0$ (depending only on $q$, $d$
and $\|C\|_{H.S.}$ --- with a continuous dependence in the last parameter)
such that, for any ${\bf x}\in\R^d$ and $t\in[\frac12,1]$,
\begin{equation}\label{ineg-fin}
{\bf TV}(\sqrt{t}F+\sqrt{1-t}{\bf x},F)\leq c_{q,d,\|C\|_{H.S.}}
\left(\beta^{-\frac{1}{N+1}}\wedge 1\right)(1+\|{\bf x}\|_1)\, (1-t)^{\frac{1}{2(2N+4)(d+1)+2}}.
\end{equation}
Here, 
$N=2d(q-1)$.
\end{lemma}

\noindent
{\it Proof}.
The proof is divided into several steps.
In what follows, we fix $t\in [\frac12,1]$ and ${\bf x}\in\R^d$ and we use the convention that
$c_{\{\cdot\}}$ denotes a constant in $(0,\infty)$ that only depends on the arguments inside the bracket and whose value is allowed to change from one line to another.

\medskip

\noindent{\it Step 1}. One has that
 $E\big[\|F\|^2_\infty\big]
\leq c_d E\big[\|F\|_2^2\big]\leq c_d\sqrt{\sum_{i,j=1}^d C(i,i)^2}\leq c_d\|C\|_{H.S.}$ 
 so that 
 $E\big[\|F\|_\infty\big]
 \leq c_d\sqrt{\|C\|_{H.S.}}=c_{d,\|C\|_{H.S.}}$ 
 Let $g:\R^d\to\R$ be a (smooth) test function bounded by 1. We can write, for any $M\geq 1$,
\begin{align}
&\big| E[g(\sqrt{t}F+\sqrt{1-t}{\bf x})]-E[g(F)]\big|\notag\\
&\leq
\left|
E\left[
(g{\bf 1}_{[-M/2,M/2]^d})(\sqrt{t}F+\sqrt{1-t}{\bf x})
\right]
-E\left[
(g{\bf 1}_{[-M/2,M/2]^d})(F)
\right]
\right|\notag\\
&\quad +P(\|\sqrt{t}F+\sqrt{1-t}{\bf x}\|_\infty\geq M/2)
+ P(\|F\|_\infty\geq M/2)
\notag\\
&\leq
\sup_{
\stackrel{\|\phi\|_\infty\leq 1}{{\rm supp}\phi\subset [-M,M]^d}
}\big| E[\phi(\sqrt{t}F+\sqrt{1-t}{\bf x})]-E[\phi(F)]\big|\notag\\
&\quad  +\frac{2}{M}\left(E\left[\|\sqrt{t}F+\sqrt{1-t}{\bf x}\|_\infty\right]
+E\left[\|F\|_\infty\right]\right)\notag\\
&\leq\!\!\!\!\!\!
\sup_{
\stackrel{\|\phi\|_\infty\leq 1}{{\rm supp}\phi\subset [-M,M]^d}
}\big| E[\phi(\sqrt{t}F+\sqrt{1-t}{\bf
  x})]-E[\phi(F)]\big|+\frac{c_{d,\|C\|_{H.S.}}}{M}(1+\|{\bf
  x}\|_\infty).\label{limit334} 
\end{align}

\medskip

\noindent{\it Step 2}. 
Let $\phi:\R^d\to\R$ be $\mathcal{C}^\infty$ with compact support in $[-M,M]^d$ and satisfying $\|\phi\|_\infty\leq 1$.
Let $0<\alpha\leq 1$ and let $\rho:\R^d\to\R_+$ be in $\mathcal{C}^\infty_c$ and satisfying $\int_{\R^d}\rho(x)dx=1$.
Set $\rho_\alpha(x)=\frac1{\alpha^d}\rho(\frac{x}{\alpha})$.
By \cite[formula (3.26)]{nourdin2012convergence}, we have that
$\phi*\rho_\alpha$  is Lipschitz continuous with constant $1/\alpha$.
We can thus write,
\begin{align}
&\big|
E[\phi(\sqrt{t}F+\sqrt{1-t}{\bf x})]-E[\phi(F)]
\big|\notag\\
&\leq\left|
E\left[\phi*\rho_\alpha(\sqrt{t}F+\sqrt{1-t}{\bf x})-\phi*\rho_\alpha(F)\right]
\right|+
\left|
E\left[\phi(F)-\phi*\rho_\alpha(F) \right] \right|
\notag\\
& \quad  +
\left|
E\left[\phi(\sqrt{t}F+\sqrt{1-t}{\bf x})-\phi*\rho_\alpha(\sqrt{t}F+\sqrt{1-t}{\bf x}) \right] \right| \notag\\
&\leq\frac{\sqrt{1-t}}{\alpha}\big(E\big[\|F\|_\infty\big]+\|{\bf x}\|_\infty\big)+
\left|
E\left[\phi(F)-\phi*\rho_\alpha(F) \right] \right|\notag\\
&\quad +
\left|
E\left[\phi(\sqrt{t}F+\sqrt{1-t}{\bf x})-\phi*\rho_\alpha(\sqrt{t}F+\sqrt{1-t}{\bf x}) \right] \right|\notag\\
&\leq c_{d,\|C\|_{H.S.}}\frac{\sqrt{1-t}}{\alpha}\big(1+\|{\bf x}\|_1\big)+
\left|
E\left[\phi(F)-\phi*\rho_\alpha(F) \right] \right|\notag\\
& \quad +
\left|
E\left[\phi(\sqrt{t}F+\sqrt{1-t}{\bf
    x})-\phi*\rho_\alpha(\sqrt{t}F+\sqrt{1-t}{\bf x}) \right]
\right|.   \label{9} 
\end{align}
In order to estimate the two last terms in (\ref{9}),  we decompose the expectation into two parts using the identity
\[
1=\frac{\e}{t^d\det\Gamma +\e}+\frac{t^d\det\Gamma }{t^d\det\Gamma +\e},\quad\e>0.
\]

\medskip

 \noindent{\it Step 3}. For all $\e,\lambda>0$
and using (\ref{carbery}),
\begin{align*}
 & E\left[\frac{\e}{t^d\det\Gamma+\e}\right]\\
& \leq
 E\left[\frac{\e}{t^d \det\Gamma +\e}\,{\bf 1}_{\{\det\Gamma>\lambda\}}\right]
+cN\left(\frac{\lambda}{E[\det\Gamma]}\right)^{1/N}\\
& \leq \frac{\e}{t^d \lambda}+cN\left(\frac{\lambda}{\beta}\right)^{1/N}.
\end{align*}
Choosing $\lambda=\e^{\frac{N}{N+1}}\beta^{\frac{1}{N+1}}$
yields  
\[
E\left[\frac{\e}{t^d\det\Gamma+\e}\right]\leq (t^{-d}+cN)\,\left(\frac{\e}{\beta}\right)^{\frac{1}{N+1}}
\leq c_{q,d}\left(\frac{\e}{\beta}\right)^{\frac{1}{N+1}}
\quad\mbox{(recall that $t\geq\frac12$}).
\]
 As a consequence, 
\begin{align}
&
\left|
E\left[\phi(\sqrt{t}F+\sqrt{1-t}{\bf x})-\phi*\rho_\alpha(\sqrt{t}F+\sqrt{1-t}{\bf x}) \right] \right|\notag
\\
&= \left|
E\left[\left(\phi(\sqrt{t}F+\sqrt{1-t}{\bf
      x})-\phi*\rho_\alpha(\sqrt{t}F+\sqrt{1-t}{\bf x})\right)
\right. \right. \times \nonumber \\
& \quad \times  \left. \left. \left(\frac{\e}{t^d\det\Gamma+\e}+\frac{t^d\det\Gamma}{t^d\det\Gamma +\e}\right)\right]
\right|\notag\\
&\leq c_{q,d}\left(\frac{\e}{\beta}\right)^{\frac{1}{N+1}} \nonumber
\\
& \quad +
\left|
E\left[\big(
\phi(\sqrt{t}F+\sqrt{1-t}{\bf x})-\phi*\rho_\alpha(\sqrt{t}F+\sqrt{1-t}{\bf x})\big)
\,\frac{t^d\det\Gamma}{t^d\det\Gamma+\e}\right]
\right|.  \notag\\
\label{7}
\end{align}

\medskip

\noindent{\it Step 4}.
In this step we recall from \cite[page 11]{nourdin2012absolute} the following integration by parts formula (\ref{eq1}).
Let $h:\R^d\to\R$ be a function in $\mathcal{C}^\infty$ with compact support, and consider a random variable $W\in \mathbb{D}^\infty$. Consider the Poisson kernel in $\mathbb{R}^d$, defined as the solution to the equation $\Delta Q_d=\delta_0$. We know that 
$Q_2(x)= c_2 \log |x|$ and  $Q_d(x) =c_d |x|^{2-d}$ for $d\neq 2$.
We have the following identity
\begin{align}  \label{eq1}
&  E[W\det \Gamma \,h(\sqrt{t}F+\sqrt{1-t}{\bf x})] \nonumber \\
& = \frac{1}{\sqrt{t}}\sum_{i=1}^d E\left[ A_{i}(W)\int_{\mathbb{R}^d} h(y) \partial _i Q_d(\sqrt{t}F+\sqrt{1-t}{\bf x}-y)dy \right],
\end{align}
where
\begin{align*}
A_{i}(W)
&=-\sum_{a=1}^{d}\big(
\langle D(W({\rm Adj}\Gamma)_{a,i}),DF_{a}\rangle_\HH+({\rm Adj}\Gamma)_{a,i}WLF_{a}
\big),
\end{align*}
with ${\rm Adj}$ the usual adjugate matrix operator.

\medskip

\noindent{\it Step 5}.
Let us apply the identity (\ref{eq1}) to the function $h=\phi-\phi*\rho_\alpha$ and to the random variable $W=W_{\e}=\frac{ 1}{t^d\det\Gamma +\e}\in \mathbb{D}^\infty$; 
we obtain
\begin{align}
 &E\left[(\phi-\phi*\rho_\alpha)(\sqrt{t}F+\sqrt{1-t}{\bf x})\frac{t^d\det\Gamma}{t^d\det\Gamma +\e}\right]\notag\\
&=t^{d-\frac12}\sum_{i=1}^d E\left[ A_{i}(W_{\e})\int_{\mathbb{R}^d} (\phi-\phi*\rho_\alpha)({\bf y}) \partial _i Q_d(\sqrt{t}F+\sqrt{1-t}{\bf x}-{\bf y})d{\bf y} \right].
\label{claim2}
\end{align}
From the hypercontractivity property together with the equality
\begin{align*}
A_{i}(W_{\e})&=\sum_{a=1}^{d}\Bigg\{ -\frac{ 1}{t^d\det\Gamma +\e}
\left(
\langle D({\rm Adj}\Gamma)_{a,i},DF_{a}\rangle_\HH-({\rm Adj}\Gamma)_{a,i}     LF_{a}\right)\\
&\quad +\frac {1} { ( t^d\det\Gamma +\e )^2}  ({\rm Adj}\Gamma)_{a,i}
\langle D(\det\Gamma),DF_{a}\rangle_\HH  \Bigg\},
\end{align*}
one immediately deduces the existence of $c_{q,d,\|C\|_{H.S.}}>0$ such that
\[
E[A_{i}(W_{\e})^2] \leq c_{q,d,\|C\|_{H.S.}} \,\e^{-4}.
\]
On the other hand, in \cite[page 13]{nourdin2012absolute} it is shown that there exists $c_d>0$  such that, for any $R>0$ and 
$u\in\R^d$,
\[
\left|\int_{\mathbb{R}^d} (\phi-\phi*\rho_\alpha)({\bf y}) \partial _i Q_d({\bf u}-{\bf y})d{\bf y}\right|\leq c_d \left( R+ \alpha+ \alpha R^{-d} (\|{\bf u}\|_1+M)^d \right).
\]
Substituting this estimate into (\ref{claim2}) and assuming that $M\geq 1$,  yields
\begin{align*}
&   \left|E\left[(\phi-\phi*\rho_\alpha)(\sqrt{t}F+\sqrt{1-t}{\bf
      x})\frac{t^d\det\Gamma}{t^d\det\Gamma +\e}\right]  \right|\\
& \leq
 c_{q,d,\|C\|_{H.S.}}\,\e^{-2} \left( R+ \alpha+ \alpha R^{-d}  (M+\|{\bf x}\|_1)^d  \right).
\end{align*}
Choosing $R=\alpha^{\frac 1{d+1}} (M+\|{\bf x}\|_1)^{\frac d {d+1}}$ and assuming $\alpha \leq 1$, we obtain
 \begin{align} \label{eq2}
&  \left|E\left[(\phi-\phi*\rho_\alpha)(\sqrt{t}F+\sqrt{1-t}{\bf
      x})\frac{t^d\det\Gamma}{t^d\det\Gamma+\e}\right]  \right|\nonumber\\
& \leq
c_{q,d,\|C\|_{H.S.}}\,\e^{-2}  \alpha^{\frac 1{d+1}} (M+\|{\bf x}\|_1)^{\frac {d}{d+1}}.
 \end{align}
It is worthwhile noting that the inequality (\ref{eq2}) is valid for any $t\in[\frac12,1]$, in particular for $t=1$.

\medskip

\noindent{\it Step 6}. From (\ref{9}),  (\ref{7}) and (\ref{eq2}) we obtain
\begin{align*}
&|E[\phi(\sqrt{t}F+\sqrt{1-t}{\bf x})]-E[\phi(F)] |\\
&\leq c_{d,\|C\|_{H.S.}}\frac{\sqrt{1-t}}{\alpha}(1+\|{\bf x}\|_1)+c_{q,d}
\left(\frac{\e}{\beta}\right)^{\frac{1}{N+1}} \\
& \quad +c_{q,d,\|C\|_{H.S.}}\,\e^{-2}\alpha^{\frac{1}{d+1}}(M+\|{\bf x}\|_1)^{\frac{d}{d+1}}.
\end{align*}
By plugging this inequality into (\ref{limit334}) we  thus obtain that, for every $M\geq 1$, $\e>0$ and $0<\alpha \leq 1$:
\begin{align}
&{\bf TV}(\sqrt{t}F+\sqrt{1-t}{\bf x},F)\notag\\
&\leq c_{d,\|C\|_{H.S.}}\frac{\sqrt{1-t}}{\alpha}(1+\|{\bf x}\|_1)+c_{q,d}
\left(\frac{\e}{\beta}\right)^{\frac{1}{N+1}}+c_{q,d,\|C\|_{H.S.}}\e^{-2}\alpha^{\frac{1}{d+1}}(M+\|{\bf x}\|_1)^{\frac{d}{d+1}}\notag\\
&\quad +
\frac{c_{d,\|C\|_{H.S.}}}{M}(1+\|{\bf x}\|_1)\notag\\
&\leq c_{q,d,\|C\|_{H.S.}}(1+\|{\bf x}\|_1)\left(\beta^{-\frac{1}{N+1}}\wedge 1\right)\left\{\frac{\sqrt{1-t}}{\alpha}+
\e^{\frac{1}{N+1}}+\frac{\alpha^{\frac{1}{d+1}}M}{\e^2}+
\frac{1}{M}\right\}\notag\\
\label{optimize1}
\end{align}
Choosing $M=\e^{-\frac{1}{N+1}}$, $\e=\alpha^{\frac{N+1}{(2N+4)(d+1)}}$ and $\alpha=(1-t)^{\frac{(2N+4)(d+1)}{2((2N+4)(d+1)+1)}}$, one obtains the desired conclusion (\ref{ineg-fin}).
\qed

\subsection{Proof of Theorem \ref{t:e4m-gene}} In the proof of \cite[Theorem 4.3]{NRaop}, the following two facts have been shown:
\begin{align*}
E\left[\left(C_n(j,k)-\frac1{q_j}\langle DF_{j,n},DF_{k,n}\rangle_\HH\right)^2\right]&\leq {\rm Cov}(F_{j,n}^2,F_{k,n}^2) - 2C_n(j,k)^2\\
E\left[\|F_n\|^4\right] - E\left[\|Z_n\|^4\right] &= \sum_{j,k=1}^d\big\{{\rm Cov}(F_{j,n}^2,F_{k,n}^2) - 2C_n(j,k)^2\big\}.
\end{align*}
As a consequence, one deduces that
\[
\sum_{j,k=1}^d E\left[\left(C_n(j,k)-\frac1{q_j}\langle DF_{j,n},DF_{k,n}\rangle_\HH\right)^2\right]\leq \Delta_n.
\]
Using Proposition \ref{p:smg}, one   infers immediately that
\begin{equation}
\tau_{F_n}^{j,k}({\bf x}):=\frac{1}{q_j}E[\langle DF_{j,n},DF_{k,n}\rangle_\HH|F_n={\bf x}],\quad j,k=1,\ldots,d,
\label{liensko}
\end{equation}
defines a Stein's matrix for $F_n$, which moreover satisfies the relation
\begin{equation}
\sum_{j,k=1}^d E\left[\left(C(j,k)-\tau_{F_n}^{j,k}(F_n)\right)^2\right]
\leq\Delta_n.
\label{liendelta}
\end{equation}
Now let $\Gamma_n$ denote the Malliavin matrix of $F_n$. Thanks to \cite[Lemma 6]{NuOrt}, we know that, for any $i,j=1,\ldots,d$,
\begin{equation}\label{537}
\langle DF_{i,n},DF_{j,n}\rangle_\HH\overset{L^2(\sigma({\bf G}))}{\longrightarrow}\sqrt{q_iq_j}C(i,j)\quad\mbox{as $n\to\infty$}.
\end{equation}
Since $\langle DF_{i,n},DF_{j,n}\rangle_\HH$ lives in a finite sum of chaoses (see e.g. \cite[Chapter 5]{NP11}) and is bounded in $L^2(\sigma({\bf G}))$,
we can again apply the hypercontractive estimate \eqref{EQ : hyperC} to deduce that $\langle DF_{i,n},DF_{j,n}\rangle_\HH$ is actually bounded in
$L^p(\sigma({\bf G}))$ for every $p\geq 1$, so that the convergence in (\ref{537}) is in the sense of any of the spaces $L^p(\sigma({\bf G}))$.
As a consequence,
$
E[\det\Gamma_n]\to \det C\prod_{i=1}^d q_i=:\gamma>0,
$
and there exists $n_0$ large enough so that
\[
\inf_{n\geq n_0}E[\det\Gamma_n]>0.
\]
We are now able to deduce from Lemma \ref{technical-but-nevertheless-crucial-lemma}
 the existence of two constants $\kappa>0$ and $\alpha\in(0,\frac12]$
such that, for all ${\bf x}\in\R^d$, $t\in[\frac12,1]$ and $n\geq n_0$,
\begin{align*}
{\bf TV}(\sqrt{t}F_n+\sqrt{1-t}{\bf x},F_n)&\leq \kappa (1+\|{\bf x}\|_1)(1-t)^{\alpha}.
\end{align*}
This means that relation (\ref{ineq-ivan2-multi}) is satisfied uniformly on $n$.
Concerning (\ref{needed-multi}), again by hypercontractivity and using the
representation (\ref{liensko}), one has that, for all $\eta>0$,
\begin{align*}
  \sup_{n\geq 1} E\big[|\tau_{F_n}^{j,k}(F_n)|^{\eta+2}\big]<\infty,\quad j,k=1,\ldots,d.
\end{align*}
Finally, since $\Delta_n\to 0$ and because (\ref{liendelta}) holds true, the condition (\ref{lessthanone}) is satisfied for $n$ large enough.
The proof of (\ref{entropy}) is concluded by applying Theorem \ref{main-multi}.

\subsection{Proof of Corollary \ref{c:co}}

In view of Theorem \ref{t:e4m-gene}, one has only to prove that (b) implies (a). This is an immediate consequence of the fact that the covariance $C_n$ converges to $C$, and that the sequence $\{F_n\}$ lives in a finite sum of Wiener chaoses. Indeed, by virtue of  \eqref{EQ : hyperC} one has that
\[
\sup_{n\geq 1} E \left[ \|F_n \|^p  \right] <\infty, \quad \forall p\geq 1,
\]
yielding in particular that, if $F_n$ converges in distribution to $Z$, then  $E\|F_n \|^4 \to E\|Z \|^4$ or, equivalently, $\Delta_n \to 0$. The proof is concluded.

\bigskip

\noindent{\bf Acknowledgement}. We heartily thank Guillaume Poly for
several crucial inputs that helped us achieving the proof of Theorem
\ref{t:e4m-gene}. We are grateful to Larry Goldstein and Oliver
Johnson for useful discussions. Ivan Nourdin is partially supported by the ANR grant ANR-10-BLAN-0121.

\bibliographystyle{abbrv}

\begin{thebibliography}{}

\end{thebibliography}


\begin{thebibliography}{10}

\bibitem{logsob_book}
C.~An{\'e}, S.~Blach{\`e}re, D.~Chafa{\"{\i}}, P.~Foug{\`e}res, I.~Gentil,
  F.~Malrieu, C.~Roberto, and G.~Scheffer.
\newblock {\em Sur les in\'egalit\'es de {S}obolev logarithmiques}, volume~10
  of {\em Panoramas et Synth\`eses [Panoramas and Syntheses]}.
\newblock Soci\'et\'e Math\'ematique de France, Paris, 2000.
\newblock With a preface by Dominique Bakry and Michel Ledoux.

\bibitem{artstein2004solution}
S.~Artstein, K.~Ball, F.~Barthe, and A.~Naor.
\newblock Solution of {S}hannon's problem on the monotonicity of entropy.
\newblock {\em J. Amer. Math. Soc.}, 17(4):975--982,
  2004.

\bibitem{ABBNptrf}
S.~Artstein, K.~M. Ball, F.~Barthe, and A.~Naor.
\newblock On the rate of convergence in the entropic central limit theorem.
\newblock {\em Probab. Theory Related Fields}, 129(3):381--390, 2004.

\bibitem{BakryEmery1985}
D.~Bakry and M.~\'Emery.
\newblock Diffusions hypercontractives.
\newblock In {\em Lecture Notes in Math.}, number 1123 in S{\'e}minaire de
  Probabilit{\'e}s XIX, pages 179--206. Springer, 1985.

\bibitem{BaBaNa03}
K.~Ball, F.~Barthe, and A.~Naor.
\newblock Entropy jumps in the presence of a spectral gap.
\newblock {\em Duke Math. J.}, 119(1):41--63, 2003.

\bibitem{BaNg}
K.~Ball and V.~Nguyen.
\newblock Entropy jumps for random vectors with log-concave density and
  spectral gap.
\newblock Preprint, arxiv:1206.5098v3, 2012.

\bibitem{BaJoKoMa10}
A.~D. Barbour, O.~Johnson, I.~Kontoyiannis, and M.~Madiman.
\newblock Compound {P}oisson approximation via information functionals.
\newblock {\em Electron. J. Probab.}, 15:1344--1368, 2010.

\bibitem{BA86}
A.~R. Barron.
\newblock Entropy and the central limit theorem.
\newblock {\em Ann. Probab.}, 14(1):336--342, 1986.

\bibitem{BCG2012}
S.~G. Bobkov, G.~P. Chistyakov, and F.~G\"otze.
\newblock Fisher information and the central limit theorem.
\newblock Preprint, arXiv:1204.6650v1, 2012.

\bibitem{Br82}
L.~D. Brown.
\newblock A proof of the central limit theorem motivated by the
  {C}ram{\'e}r-{R}ao inequality.
\newblock In {\em Statistics and probability: essays in honor of {C}. {R}.
  {R}ao}, pages 141--148. North-Holland, Amsterdam, 1982.

\bibitem{CW}
A.~Carbery and J.~Wright.
\newblock Distributional and {$L^q$} norm inequalities for polynomials over
  convex bodies in {$\Bbb R^n$}.
\newblock {\em Math. Res. Lett.}, 8(3):233--248, 2001.

\bibitem{carlen1991entropy}
E.~Carlen and A.~Soffer.
\newblock Entropy production by block variable summation and central limit
  theorems.
\newblock {\em Commun. Math. Phys.}, 140(2):339--371, 1991.

\bibitem{ChPtrf2009}
S.~Chatterjee.
\newblock Fluctuations of eigenvalues and second order {P}oincar{\'e}
  inequalities.
\newblock {\em Probab. Theory Related Fields}, 143(1-2):1--40, 2009.

\bibitem{MR2453473}
S.~Chatterjee and E.~Meckes.
\newblock Multivariate normal approximation using exchangeable pairs.
\newblock {\em ALEA Lat. Am. J. Probab. Math. Stat.}, 4:257--283, 2008.

\bibitem{ChGoSh11}
L.~H.~Y. Chen, L.~Goldstein, and Q.-M. Shao.
\newblock {\em Normal approximation by {S}tein's method}.
\newblock Probability and its Applications (New York). Springer, Heidelberg,
  2011.

\bibitem{Cs}
I.~Csisz\'ar.
\newblock Informationstheoretische {K}onvergenzbegriffe im {R}aum der
  {W}ahrscheinlichkeitsverteilungen.
\newblock {\em Magyar Tud. Akad. Mat. Kutat{\'o} Int. K{\"o}zl.}, 7:137--158,
  1962.

\bibitem{DNN}
A.~Deya, S.~Noreddine, and I.~Nourdin.
\newblock Fourth moment theorem and q-{B}rownian chaos.
\newblock {\em Comm. Math. Phys., to appear}, 02 2012.

\bibitem{dudley_book}
R.~M. Dudley.
\newblock {\em Real analysis and probability}.
\newblock The Wadsworth \& Brooks/Cole Mathematics Series. Wadsworth \&
  Brooks/Cole Advanced Books \& Software, Pacific Grove, CA, 1989.

\bibitem{Goetze} F. G\"{o}tze (1991). On the rate of convergence in the multivariate CLT.
\textit{Ann. Probab.}, {19}(2), 724-739.

\bibitem{HLN} Y. Hu, F. Lu and D. Nualart (2013). Convergence of densities of some functionals of Gaussian processes. Preprint.

\bibitem{Jo04}
O.~Johnson.
\newblock {\em Information theory and the central limit theorem}.
\newblock Imperial College Press, London, 2004.

\bibitem{MR2128239}
O.~Johnson and A.~Barron.
\newblock Fisher information inequalities and the central limit theorem.
\newblock {\em Probab. Theory Related Fields}, 129(3):391--409, 2004.

\bibitem{johnson2001entropy}
O.~Johnson and Y.~Suhov.
\newblock Entropy and random vectors.
\newblock {\em J. Statist. Phys.}, 104(1):145--165, 2001.

\bibitem{KNPS}
T.~Kemp, I.~Nourdin, G.~Peccati, and R.~Speicher.
\newblock Wigner chaos and the fourth moment.
\newblock {\em Ann. Probab.}, 40(4):1577--1635, 2012.

\bibitem{KullIEEE}
S.~Kullback.
\newblock A lower bound for discrimination information in terms of variation.
\newblock {\em IEEE Trans. Info. Theory}, 4, 1967.

\bibitem{kumar2009stein}
S.~Kumar~Kattumannil.
\newblock On {S}tein's identity and its applications.
\newblock {\em Statist. Probab. Lett.}, 79(12):1444--1449, 2009.

\bibitem{LedouxAOP}
M.~Ledoux.
\newblock Chaos of a {M}arkov operator and the fourth moment condition.
\newblock {\em Ann. Probab.}, 40(6):2439--2459, 2012.

\bibitem{LS11c}
C.~Ley and Y.~Swan.
\newblock Local Pinsker inequalities via Stein's discrete density approach
\newblock {\em IEEE Trans. Info. Theory}, to appear, 2013.

\bibitem{LS12a}
C.~Ley and Y.~Swan.
\newblock Stein's density approach and information inequalities.
\newblock {\em Electron. Comm. Probab.}, 18(7):1--14, 2013.

\bibitem{LI59}
J.~V. Linnik.
\newblock An information-theoretic proof of the central limit theorem with
  {L}indeberg conditions.
\newblock {\em Theor. Probability Appl.}, 4:288--299, 1959.

\bibitem{MP11}
D.~Marinucci and G.~Peccati.
\newblock {\em Random fields on the sphere : Representation, limit theorems and
  cosmological applications}, volume 389 of {\em London Mathematical Society
  Lecture Note Series}.
\newblock Cambridge University Press, Cambridge, 2011.

\bibitem{NouBookFBM}
I.~Nourdin.
\newblock {\em Selected Aspects of Fractional Brownian Motion}.
\newblock Springer-Verlag, 2012.

\bibitem{nourdin2012absolute}
I.~Nourdin, D.~Nualart, and G.~Poly.
\newblock Absolute continuity and convergence of densities for random vectors
  on {W}iener chaos.
\newblock {\em Electron. J. Probab.}, 18(22):1--19, 2012.

\bibitem{np-jfa} I. Nourdin and G. Peccati (2010). Cumulants on the
  Wiener space. {\it J. Funct. Anal.} {258},
  3775-3791. 

\bibitem{NPalea}
I.~Nourdin and G.~Peccati.
\newblock Universal {G}aussian fluctuations of non-{H}ermitian matrix
  ensembles: from weak convergence to almost sure {CLT}s.
\newblock {\em ALEA Lat. Am. J. Probab. Math. Stat.}, 7:341--375, 2010.

\bibitem{NP11}
I.~Nourdin and G.~Peccati.
\newblock {\em Normal approximations with Malliavin calculus : from Stein's
  method to universality}.
\newblock Cambridge Tracts in Mathematics. Cambridge University Press, 2012.

\bibitem{npr-jfa} I. Nourdin, G. Peccati and G. Reinert (2009). Second order Poincar\'e inequalities and CLTs on Wiener space. {\it J. Funct. Anal.} {\bf 257}, 593-609.

\bibitem{NPRaop}
I.~Nourdin, G.~Peccati, and G.~Reinert.
\newblock Invariance principles for homogeneous sums: universality of
  {G}aussian {W}iener chaos.
\newblock {\em Ann. Probab.}, 38(5):1947--1985, 2010.

\bibitem{NPRihp}
I.~Nourdin, G.~Peccati, and A.~R{{\'e}}veillac.
\newblock Multivariate normal approximation using {S}tein's method and
  {M}alliavin calculus.
\newblock {\em Ann. Inst. Henri Poincar{\'e} Probab. Stat.}, 46(1):45--58,
  2010.

\bibitem{nourdin2012convergence}
I.~Nourdin and G.~Poly.
\newblock Convergence in total variation on {W}iener chaos.
\newblock {\em Stochastic Process. Appl.}, to appear, 2012.

\bibitem{NRaop}
I.~Nourdin and J.~Rosi\'nski.
\newblock Asymptotic independence of multiple {W}iener-{I}t{\^o} integrals and
  the resulting limit laws.
\newblock {\em Ann. Probab.}, to appear, 2013.

\bibitem{Nu06}
D.~Nualart.
\newblock {\em The {M}alliavin calculus and related topics}.
\newblock Probability and its Applications (New York). Springer-Verlag, Berlin,
  second edition, 2006.

\bibitem{NuOrt}
D.~Nualart and S.~Ortiz-Latorre.
\newblock Central limit theorems for multiple stochastic integrals and
  {M}alliavin calculus.
\newblock {\em Stochastic Process. Appl.}, 118(4):614--628, 2008.

\bibitem{NunuGio}
D.~Nualart and G.~Peccati.
\newblock Central limit theorems for sequences of multiple stochastic
  integrals.
\newblock {\em Ann. Probab.}, 33(1):177--193, 2005.

\bibitem{park2012equivalence}
S.~Park, E.~Serpedin, and K.~Qaraqe.
\newblock On the equivalence between {S}tein and de {B}ruijn identities.
\newblock {\em IEEE Trans. Info. Theory}, 58(12):7045--7067, 2012.

\bibitem{PinskerBook}
M.~S. Pinsker.
\newblock {\em Information and information stability of random variables and
  processes}.
\newblock Translated and edited by Amiel Feinstein. Holden-Day Inc., San
  Francisco, Calif., 1964.

\bibitem{MR2573554}
G.~Reinert and A.~Roellin.
\newblock Multivariate normal approximation with {S}tein's method of
  exchangeable pairs under a general linearity condition.
\newblock {\em Ann. Probab.}, 37(6):2150--2173, 2009.

\bibitem{sason2012entropy}
I.~Sason.
\newblock Entropy bounds for discrete random variables via maximal coupling.
\newblock {\em IEEE Trans. Info. Theory}, to appear, 2013.

\bibitem{Sa12}
I.~Sason.
\newblock Improved lower bounds on the total variation distance for
the Poisson approximation.
\newblock  {\em Statist. Probab. Lett.}, 83(10):2422--2431, 2013.

\bibitem{Sh75}
R.~Shimizu.
\newblock On {F}isher's amount of information for location family.
\newblock In {\em A Modern Course on Statistical Distributions in Scientific
  Work}, pages 305--312. Springer, 1975.

\bibitem{St86}
C.~Stein.
\newblock {\em Approximate computation of expectations}.
\newblock Institute of Mathematical Statistics Lecture Notes---Monograph
  Series, 7. Institute of Mathematical Statistics, Hayward, CA, 1986.

\bibitem{TalagrandGFA}
M.~Talagrand.
\newblock Transportation cost for {G}aussian and other product measures.
\newblock {\em Geom. Funct. Anal.}, 6(3):587--600, 1996.

\bibitem{viensSPA}
F.~G. Viens.
\newblock Stein's lemma, {M}alliavin calculus, and tail bounds, with
  application to polymer fluctuation exponent.
\newblock {\em Stochastic Process. Appl.}, 119(10):3671--3698, 2009.

\end{thebibliography}

\end{document}